\documentclass[12pt]{amsart}
\usepackage{amssymb,amsfonts,amsthm,amsaddr}
\usepackage{graphicx}
\usepackage{color} 
\usepackage{bm}
\newtheorem{thm}{Theorem}[section]
\newtheorem{cor}[thm]{Corollary}

\newtheorem{pro}[thm]{Proposition}

\theoremstyle{definition}

\theoremstyle{remark}

\numberwithin{equation}{section}

\begin{document}
\title[Stochastic energy exchange model]{From billiards to
  thermodynamic laws: Stochastic energy exchange model }
\author[Yao Li]{Yao Li}
\address{Yao Li: Department of Mathematics and Statistics, 
University of Massachusetts Amherst, Amherst, MA, 01002, USA}
\email{yaoli@math.umass.edu}

\author{Lingchen Bu}
\address{Lingchen Bu: Department of Mathematics and Statistics, 
University of Massachusetts Amherst, Amherst, MA, 01002, USA}
\email{lingchen.l.bu@gmail.com}

\begin{abstract}
  This paper studies a billiards-like microscopic heat conduction
  model, which describes the dynamics of gas molecules in a long tube
  with thermalized boundary. We numerically investigate the law of
  energy exchange between adjacent cells. A stochastic
  energy exchange model that preserves these properties is
  then derived. We further numerically justified that the stochastic energy
  exchange model preserves the ergodicity and the thermal conductivity
  of the original billiard model. 
\end{abstract}

\maketitle

{\bf
The derivation of macroscopic thermodynamic laws from microscopic
Hamiltonian dynamics is a century-old challenge dating back to
Boltzmann. In this paper, we use a billiards-like
Hamiltonian model to study the microscopic heat conduction of gas
molecules in a long tube. After a series of numerical simulation, a mathematically tractable
stochastic energy exchange model is derived. We further show that many key properties of
the original deterministic problem are preserved by this stochastic
energy exchange model. This result opens the door to rigorous
justifications of many interesting problems in nonequilibrium statistical
mechanics. In our forthcoming papers, we will study the ergodicity, 
mesoscopic limit, and macroscopic thermodynamic laws of the stochastic energy
exchange model derived in this paper.

}
\section{Introduction}
in a nonequilibrium setting, it is very difficult
to prove that deterministic interactions among gas molecules or
crystal structures leads to macroscopic thermodynamic laws such as Fourier's law \cite{bonetto2000fourier}. From a dynamical systems
point of view, deterministic interactions of gas molecules can be
modeled as elastic collisions of particles. However, studying many-particle
billiard systems is a very difficult, with only limited known result
\cite{bunimovich2013hard, simanyi2003proof, simanyi1999hard}. It is
almost impossible to provide any mathematical derivation of
macroscopic thermodynamic laws by working on a many-particle billiard
model. Most known results that connect dynamical billiards and
thermodynamics are for one particle model, noninteracting particles, and weakly interacting
particles \cite{dolgopyat2016nonequilibrium, boldrighini1983boltzmann,
leonel2016thermodynamics, cook2012random, chumley2013billiards}.

On the other hand, there are also many stochastic microscopic heat
conduction models which assume some randomness among particle
interactions or energy transports. Stochastic interacting particle models are known to be
more tractable. There have been numerous known results about 
nonequilibrium steady-state, entropy production rate, fluctuation
theorem, thermal conductivity, and Fourier's law for various
stochastic models \cite{eckmann2006nonequilibrium, grigo2012mixing,
  kipnis1982heat, liverani2011toward, derrida1998exactly,
  eckmann1999entropy, rey2001exponential, rey2002fluctuations}. Therefore, it is tempting to reduce a deterministic
heat conduction model to a stochastic one. One example is
\cite{2014LiYoungNonlinearity}, in which a particle-disk collision model is well-approximated
by a stochastic particle system. 

This paper serves as the first paper of a sequel. In this sequel, we
will investigate how thermodynamic laws are derived from
billiards-like deterministic dynamics. Different from many pioneering
work, the billiard system in our study consists of a large number of
strongly interacting particles, which makes a rigorous study extremely
difficult. The goal of this paper is to numerically
justify that the deterministic dynamics in a billiards-like dynamical
system is well-approximated by a Markovian energy exchange
process. We carry out a series of numerical simulations to determine the rule of
stochastic energy exchanges. Then we will use computer-assisted method
proposed in \cite{li2017numerical} to justify that the resultant
stochastic energy exchange model preserves many key properties of the
original deterministic dynamics, such as ergodicity and thermal
conductivity.

Consider many gas particles in a long and thin tube as in Figure
\ref{fig1} (top). Assume further
that two ends of the tube is connected to heat baths with different
temperatures. For the sake of simplicity, we assume gas particles only
do free motion and elastic collisions. When a particle hits the left
(or right) boundary, a random particle drawn from a Boltzmann
distribution is chosen to collide with this particle. Besides that,
everything else is purely deterministic. Needless to say, any analysis
of such a strongly interacting multi-body problem is extremely
difficult. On the other hand, usually
a gas particle collides with other particles very frequently. At
ambient pressure, the mean free path of a gas molecule is as short as
$68$ nm \cite{jennings1988mean}. Therefore, motivated by earlier studies
\cite{bunimovich1992ergodic, gaspard2008heat, gaspard2008heat2, gaspard2008derivation}, we propose to
simplify the model by localizing gas molecules in a chain of
cells as in Figure \ref{fig1} (bottom). Disk-shaped particles are assumed to be trapped in those 2D cells, each of which is a chaotic billiard table.  Adjacent cells are connected by a ``gate''. Particles can
not pass the ``gate'' but can collide through it. This gives the {\it
  nonequilibrium billiard model} introduced in Section 2. Throughout
this paper, we assume that each cell contains $M$ particles. $M$ is a
fixed finite number.

\begin{figure}[h]
\centerline{\includegraphics[width = \textwidth]{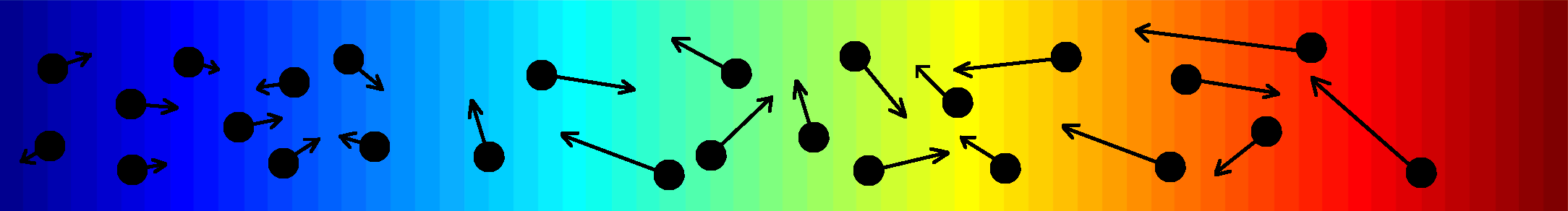}}
\centerline{\includegraphics[width = \textwidth]{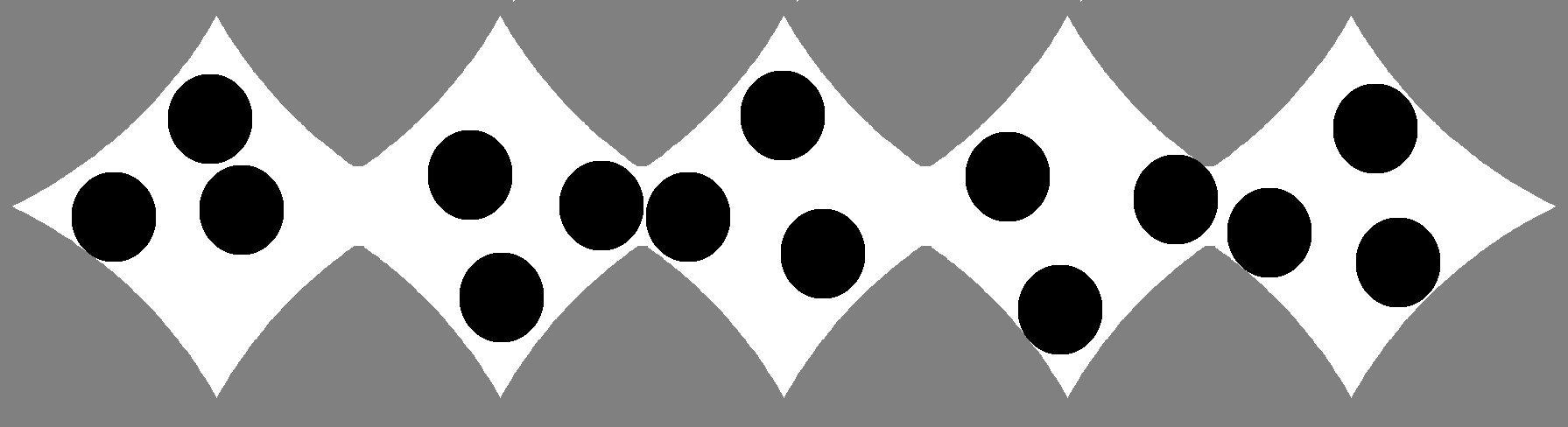}}
\caption{Top: Gas molecules in a tube. Bottom: Particles are trapped
  in a chain of cells. Particles can collide through the ``gate''.}
\label{fig1}
\end{figure}

Since the motion of gas particles is highly chaotic, a particle has
quick loss of memory. In fact, it is known that a chaotic billiard
system usually have good statistical properties
\cite{bunimovich1981statistical, bunimovich1991statistical,
  chernov2000decay, chernov2005billiards}. That is why we believe a 
Markov model should well approximate such a billiard model. In order
to make the billiard model mathematics tractable, instead of modeling each particle, we
choose to look for a Markov process that describes the time evolution
of the total energy stored in each cell. When particles in neighbor cells
collide, we let the corresponding cells exchange a certain random
amount of energy. This Markov process is called the {\it stochastic energy
  exchange model}. Due to the significant difficulty
of studying a multibody billiard system, a rigorous derivation of this Markov process is not
possible. Instead, we use numerical simulation to justify this model
reduction. 

One needs to answer two questions in order to find the rule of
stochastic energy exchange model: {\it when} and {\it how} two
particles in neighbor cells exchange energy. To answer these
questions, a series of numerical studies are carried out in Section
3. We find that the time distribution to the next collision between particles from
adjacent cells is not visually distinguishable from an exponential
 distribution. Therefore, one exponential clock should be associated
with each adjacent pair of cells. The rate of this exponential clock is
approximately square root of the minimum of two energies at the tail,
as a low total energy in a cell is always associated with a long time interval between
two energy exchanges. The rule of energy exchange is more
complicated. But after some numerical simulations, we find a simple rule
that can preserve qualitative properties that we are interested
in. The energy of the particle that participates the collision
satisfies a Beta distribution. After the collision, the energy is
uniformly redistributed. Although this is not exactly a precise rule
of energy distribution at a collision, it is simple enough while preserves the tail distribution
on the low energy side, which determines asymptotic properties
of the system. Although the entire reduction to a Markov process is
not rigorous, we attempt to provide as much mathematical justification
as possible based on several assumptions that heuristically should be
valid for sufficiently chaotic interacting particle system. 

The stochastic energy exchange model is then summarized in the end of
Section 3. Then we compared those two models from the aspects of
ergodicity and thermal conductivity. In Section 4, we numerically
show that both models has polynomial ergodicity $\sim t^{-2M}$,
where $M$ is the number of particles in each cell. Due to the
significant difficult of direct simulations, for the billiard
model, we use Monte Carlo simulation to compute the first passage time
to a ``high-energy state'', same as done in \cite{2015ChaosLi}. For
the stochastic energy exchange model, we adopt a computer assisted method
proposed in \cite{li2017numerical}. Some key estimations regarding
return times are obtained numerically, while other ingredients are rigorous. We
show that the stochastic energy exchange process admits a unique
nonequilibrium steady state. The speed of convergence to this steady
state, and the speed of correlation decay, are both polynomial. Finally, in Section 5, we compute
the thermal conductivity of the stochastic energy exchange model. We
find that the thermal conductivity is proportional to $1/N$, which is
consistent with early study of the billiard model in
\cite{gaspard2008heat, gaspard2008heat2}
(with one particle in each cell).

\section{Nonequilibrium billiard model for microscopic heat
  conduction}
\label{billiards}
As discussed in the introduction, it is difficult to study the
dynamics when a large number of gas molecules moving and interacting in a
tube. Since the mean free path of a gas
molecule is very short, we ``localize'' gas molecules into a chain of
cells to simplify the dynamics. The precise description of this locally confined particle
system is as follows. 

Consider an 1D chain of $N$ connected billiard tables in
$\mathbb{R}^{2}$, denoted by $\Omega_{1}, \cdots, \Omega_{N}$. Each table is
a subset of $\mathbb{R}^{2}$ whose boundary is formed by finitely many piecewise
$C^{3}$ curves. Neighboring billiard tables are connected by one or
finitely many ``bottleneck'' openings. The first and the last tables are
connected to the heat bath. The interaction with the heat bath will be
described later. 

Let $M$ be a positive integer that is fixed throughout this
section. Assume inside each billiard table there are $M$ rigid moving
disks with mass $2$ and radius $r$. Each disk-shaped particle moves freely until
it hits the boundary of the billiard table, or other particles. The
configuration of a state of particles in the $n$-th cell is denoted by
$(\mathbf{x}^{n}_{1}, \mathbf{v}^{n}_{1}, \cdots, \mathbf{x}^{n}_{M},
\mathbf{v}^{n}_{M})$, where $\mathbf{x}^{n}_{k}\in\mathbb{R}^2$ and
$\mathbf{v}^{n}_{k}\in \mathbb{R}^2$ are position (of the center)
and velocity of the $k$-th
particle in the $n$-th cell respectively. We assume the following for
this billiard system.

\begin{itemize}
  \item A particle is trapped in the cell in a way that its trajectory
    will never leave $\Omega_{n}$.
\item Particles in neighbor cells can collide with each other without
  passing through the opening between cells.
\item All collisions are elastic. Particles do not
  rotate. 
\item The billiard system is chaotic.
\item Let $\Gamma_{n} \subset \Omega_{n}$ be the collection of
  possible positions of particles in the $n$-th table. There exist
  positions $\mathbf{x}^{n}_{1}, \cdots, \mathbf{x}^{n}_{M}$ and
  $\epsilon > 0$ such that
$$
  | \mathbf{x}^{n}_{i} - \mathbf{x}^{n}_{j} | > 2r + \epsilon \mbox{
    for all } i, j = 1 \sim M, i \neq j 
$$
and
$$
  \bigcup_{i = 1}^{M}B(\mathbf{x}^{n}_{i}, R + \epsilon) \subset
  \Gamma_{n}\setminus \bigcup_{|m - n| = 1}B(\Gamma_{m}, r) \,.
$$
In other words particles in a table can be
completely out of reach by their neighbors. In addition a cell is sufficiently large such that particles won't get stuck.
\end{itemize}

Now we couple this chain with two heat baths. The temperature of two
heat baths are $T_{L}$ and $T_{R}$ respectively. We assume that the heat
bath is a billiard table with the same geometry and the same number of
moving particles. But the total
energy in the heat bath is randomly chosen. The
rule of the heat bath interaction is the following. At the beginning a
random total energy $E_{L}$ (resp. $E_{R}$) is chosen for the left
(resp. right) heat bath from the exponential distribution with
mean $T_{L}$ (resp. $T_{R}$). The initial distribution of particle
positions and velocities satisfies the conditional Liouville measure
(conditioning with the total energy $E_{L}$). This system
is evolved deterministically until the first collision between a heat bath particle and
a ``regular'' particle in the leftmost (resp. rightmost)
table. Immediately after such a collision, particles in the heat bath
are independently redistributed with a new total energy and new initial
positions/velocities, which are drawn from the same distribution. This is an 
idealized way to approximate the interaction with a heat bath that has
infinitely many particles. 

Let
$$
  \mathbf{\Omega} = \left \{  ( \mathbf{x}^{1}_{1}, \mathbf{v}^{1}_{1}, \cdots, \mathbf{x}^{1}_{M},
\mathbf{v}^{1}_{M}), \cdots, (\mathbf{x}^{N}_{1}, \mathbf{v}^{N}_{1}, \cdots, \mathbf{x}^{N}_{M},
\mathbf{v}^{N}_{M})  \right\} \subset\mathbb{R}^{4MN}
$$
be the state space of this billiard model. Let $\Phi_{t}
:\mathbf{\Omega} \rightarrow \mathbf{\Omega}$ be the flow generated by
the billiard model. It is easy to see that $\Phi_{t}$ is a piecewise
deterministic Markov process.

\section{Reduction to stochastic energy exchange model}
In order to make the nonequilibrium billiard model
introduced in Section \ref{billiards} tractable for further rigorous
studies, we need to consider the evolution of some coarse-grained variables
instead of velocities and positions of all particles. As introduced in
the introduction, we look for a Markov process that describes the time evolution
of total energy stored in each cell. 

The aim of this section is to provide numerical and mathematical justifications of such reduction
from the deterministic billiard model to a stochastic energy exchange
model. We remark that this section is {\it
  not} intended to be mathematically rigorous. In fact, any rigorous study of
a billiard system with more than one moving particle is extremely
difficult, with fairly limited known results \cite{kramli1991k,
  simanyi1992k, simanyi1999hard, simanyi2003proof}. Therefore, mathematical 
justifications in this section have to be built on various heuristic
assumptions. We provide as much mathematical justifications as
possible for each argument we raise. The conclusion is then verified by
carefully designed numerical simulations. 

In the following two subsections, we study {\it when} and {\it
  how} an energy exchange between two adjacent cells, i.e., a
collision between two particles from each cell respectively, should
happen. The setting of our numerical simulation is as follows. The
boundary of two cells is determined by $6$ circles and $2$ line
segments as seen in Figure \ref{fig:TwoCells}. In each cell, there
are $M$ particles undergoing free motion and elastic collisions. We use Monte Carlo
simulation to study the distribution of collision times and the
distribution of energy transferred during a collision. 

Cells in Figure \ref{fig:TwoCells} are designed such that all cell
boundaries are either flat or convex inwards, which makes motion of 
all particles chaotic \cite{chernov2006chaotic}. Recall that our main requirement of
the cell geometry is that it should generate a chaotic billiard
system. We expect our numerical result to be valid for any nonequilibrium billiard
model that satisfies our assumptions.  

\begin{figure}[h]
	\centering
	\includegraphics[width=0.7\linewidth]{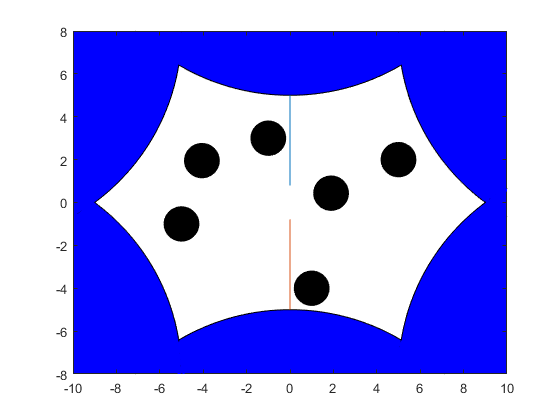}
	\caption{Setting of our numerical simulations. Two adjacent
          cells with three particles in each cell. Particles can
          collide through the gate, but can't pass through the gate. }
	\label{fig:TwoCells}
\end{figure}

\subsection{Distribution of collision time. } The first numerical result is
about the time distribution of energy exchanges between particles from
neighboring cells, called the {\it collision time}. The deterministic billiard model is
highly chaotic, which heuristically indicates a quick decay of correlation. This is
the main motivation for us to look for its Markovian
approximation. Due to the quick correlation decay, we expect collision times
between particles from adjacent cells to be close to an
inhomogeneous Poisson process. This is to say, when starting from a
fixed energy configuration, the first collision time should be
well-approximated by an exponential distribution. Further,
conditioning on the same energy configuration, the time duration between
two consecutive collisions should also satisfy an exponential
distribution with the same rate. In dynamical systems, these two
distributions are called the {\it hitting time} and the {\it return
  time} respectively. It is known that for a strongly mixing dynamical system,
those two times to an asymptotically small set coincides
\cite{haydn2005hitting, haydn2009compound}. We provide
the following simulations to study distributions of the hitting time
and the return time of the billiard model.

{\bf Rate of hitting time.}  Define the random variable
\begin{equation}
  \label{eq3-1}
\tau_c = \inf\{t>0|\text{ a cell-cell collision occurs at time}\,t\} \,.
\end{equation}

If the collision time is Poisson distributed, we should have
\begin{equation}
  \label{eq3-2}
P(\tau_c>t)=\int_{t}^{+\infty} \lambda e^{-\lambda x} \, dx
= e^{-\lambda t} \,.
\end{equation}
In other words the rate of this Poisson distribution is $\lambda = -\frac{1}{t}\log P(\tau_c>t) $.
However, the collision time is not Poisson distributed because the energy process produced by the
billiard model is clearly not Markovian. Instead, we expect the
distribution of $\tau_{c}$ to have an exponential tail. The slope of
such tail, if exists, is called the stochastic energy exchange rate. More
precisely, we are interested in $\lim_{t\to\infty} -\frac{1}{t}\log
P(\tau_c>t)$.

Consider an energy configuration $(E_{1}, E_{2})$ that corresponds to
total cell energy $E_{1}$ and $E_{2}$ in the left cell and right cell
respectively. Let $\pi$ be the Liouville measure with respect to two
cells and their $2 M$ particles, which is an
invariant measure of the billiard system involving two neighboring cells. A function $R(E_{1}, E_{2})$ is said
to be a {\it stochastic
  energy exchange rate} if 
\begin{align}
\label{eq3-3}
R(E_1,E_2) 
& = \lim_{t\to\infty} -\frac{1}{t} 
\log P_{\pi}[\tau_c >t|
\sum_{i=1}^{M}|\mathbf{v}^{1}_i|^2=E_1, \sum_{i = 1}^{M}|\mathbf{v}^{2}_j|^2=E_2] \\
& = \lim_{t\to\infty} -\frac{1}{t} 
\log P_{\pi}[\tau_c >t|\sum_{i=1}^{M}|\mathbf{v}^{1}_i(0^+)|^2=E_1,
  \sum_{i = 1}^{M}|\mathbf{v}_i(0^+)|^2=E_2 \,, \\
& \mathbf{x}^{1}_i\in \text{int}(\Gamma_1), \mathbf{x}^{2}_j\in \text{int}(\Gamma_2),
 |\mathbf{x}^{1}_i(0^{+})-\mathbf{x}^{2}_j(0^{+})|=2r \mbox{ for some } 1
  \leq i, j \leq M]
\end{align}
is well-defined. The first limit gives the tail of the
first collision time distribution when starting from a conditioning Liouville
measure (conditioning with the energy configuration $(E_{1},
E_{2})$). The second limit gives the tail of return time distribution,
when starting from the configuration corresponding to an
energy exchange event. 

Note that obviously $R(\alpha E_{1}, \alpha E_{2}) = \alpha^{1/2}
R(E_{1}, E_{2})$, we only need to simulate the case of $E_{1} + E_{2}
= 1$. Without loss of generality we assume $E_{1} \leq E_{2}$. We use
Monte Carlo simulations to compute distributions of $\tau_{c}$ for $M =
2, 3, 4$. For each $M$, we use $26$ different initial energy configurations $E_{1} =
  0.5, 0.4, 0.3, 0.2, 0.1, 0.05, 0.02, 0.01, 5\times 10^{-3}, 2 \times
  10^{-3}, 1\times 10^{-3}, 5\times 10^{-4}, 2 \times
  10^{-4}, 1\times 10^{-4}, 5\times 10^{-5}, 2 \times
  10^{-5}, 1\times 10^{-5}, 5\times 10^{-6}, 2 \times
  10^{-6}, 1\times 10^{-6}, 5\times 10^{-7}, 2 \times
  10^{-7}, 1\times 10^{-7}, 5\times 10^{-8}, 2 \times
  10^{-8}, 1\times 10^{-8}$. The initial distribution is a conditional Liouville measure, at
which the initial particle positions are uniformly distributed, and
the initial particle velocities are uniformly distributed on a
sub-manifold of $\mathbb{S}^{4M}$ such that $|\mathbf{v}^{1}_i(0^+)|^2=E_1,
  \sum_{i = 1}^{M}|\mathbf{v}_i(0^+)|^2=E_2$. The energy exchange rate
  $R(E_{1}, E_{2})$ is obtained by calculating the slope of
  distributions of $\tau_{c}$ in log-linear plots. 

Figure \ref{sampletail} shows three
  sample distribution curves of $\tau_{c}$ starting from different
  energy configurations. The probability $\mathbb{P}[\tau_{c} > t]$
  forms a straight line in the log-linear plot until there are not
  enough samples with $\tau_{c} > t$. In fact, for all energy
  configurations we have tested, one can not visually
  distinguish the distribution of $\tau_{c}$ from that of a genuine
  exponential distribution. This numerically verifies our assumption
  that the distribution of $\tau_{c}$ always has an exponential tail.

 Then we present the result $R(E_{1}, E_{2})$ versus $E_{1}$ for $M =
 2, 3, 4$, which is plotted in Figure
  \ref{exrate}. This is consistent with our numerical
finding in \cite{2015ChaosLi}. We are more interested in the scaling of $R$ as $E_{1} \rightarrow
0$. From the slope in the log-log plot in Figure 
\ref{exrate}, we can see that $R(E_{1}, E_{2}) \sim \sqrt{\min \{
  E_{1}, E_{2}\}}$ when $\min \{
  E_{1}, E_{2}\} \ll 1$ in all three cases. 

\begin{figure}[htbp]
\centerline{\includegraphics[width = \linewidth]{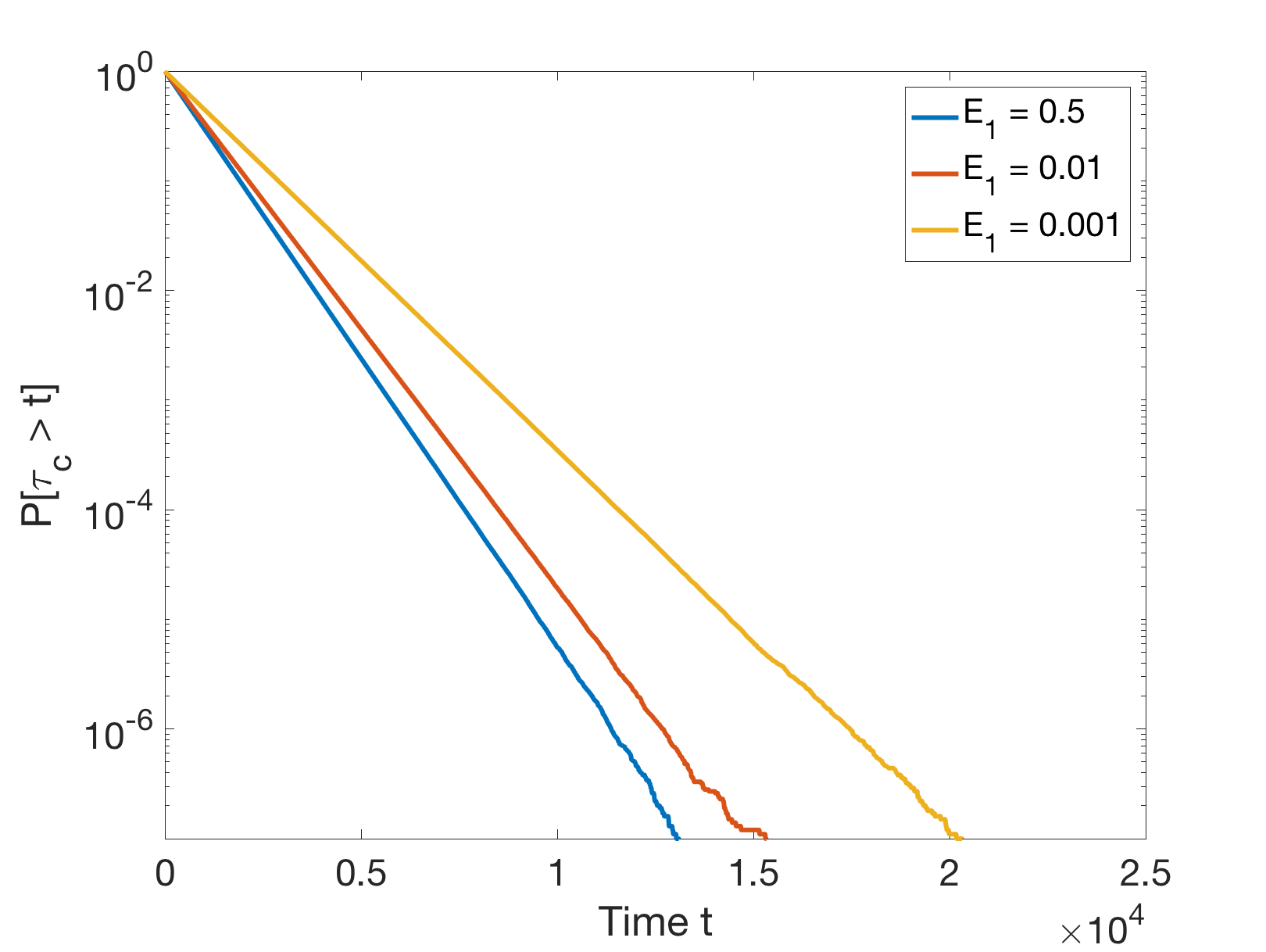}}
\caption{Three examples of distributions of $\tau_{c}$ with $E_{1} =
  0.5, 0.01$ and $0.001$. Each cell has $3$ moving particles. Each
  distribution is obtained by running $10^{8}$ independent
  trajectories until the first collision time. }
\label{sampletail}
\end{figure}

\begin{figure}[h]
\centerline{\includegraphics[width = \linewidth]{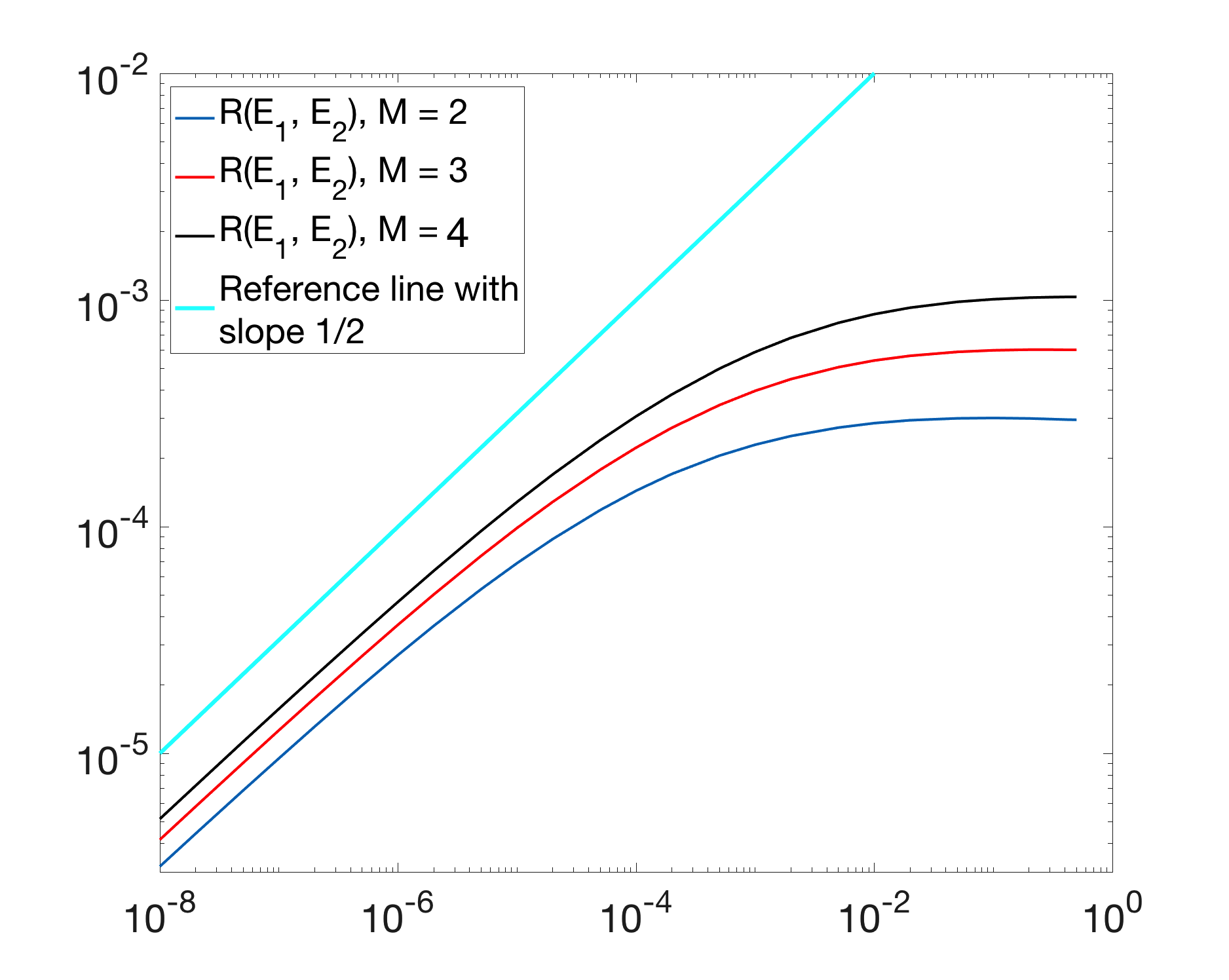}}
\caption{$R(E_{1}, E_{2})$ versus $E_{1}$ in log-log plot for $M = 2,
  3, 4$ with constraint $E_{1} + E_{2} = 1$. The initial distribution is a conditional Liouville measure
  conditioning on the energy configuration $(E_{1}, E_{2})$. For each
  $M$, $26$ different initial energy configurations $(E_{1}, E_{2})$
  with are chosen. The distribution of $\tau_{c}$ is obtained by
  running $10^{8}$ trajectories until the first collision
  time. For each energy configuration, $R(E_{1}, E_{2})$ is the slope of $\tau_{c}$ in the log-linear
  plot. }
\label{exrate}
\end{figure}

{\bf Hitting time vs. return time.} It remains to verify
that the second limit in \eqref{eq3-3} produces the
same tail. This is to say, we need to check that when conditioning on
the same energy configuration, the distribution of
time duration between two consecutive collision times, called the {\it
conditional return time}, has the same
exponential tail. Let $(E_1^{(1)}, E_2^{(1)}, t^{(1)}), (E_1^{(2)},
E_2^{(2)}, t^{(1)} + t^{(2)}), ..., (E_1^{(N)}, E_2^{(N)}, t^{(1)} +
\cdots + t^{(N)})$ be a long
trajectory sampled at collision times from a simulation starting from
$\mathbf{x}_0$. Same as in \cite{2015ChaosLi}, we expect to have a joint probability density
function $\rho_{\mathbf{x}_{0}}( E_{1}, E_{2}, t)$ about the
conditional return time and the energy configuration. Further, if
$R(E_{1}, E_{2})$ is well-defined, we should have
\begin{equation}
\label{eq3-4}
    \lim_{t\rightarrow \infty} \frac{1}{t} \log \left ( \int_{t}^{\infty}
  \rho_{\mathbf{x}_{0}}(E_{1}, E_{2}, s)|_{(E_{1}, E_{2})} \mathrm{d}s
\right )
  = R(E_{1}, E_{2}) \,.
\end{equation}

Assume $E_{1} + E_{2} = 1$. We define the following rescaled return time 
\begin{equation}
  \label{eq3-5}
\Lambda(t)=\int_{t}^{\infty} \rho_{x_0}(E,1-E,s)R(E,1-E)dEds \,.
\end{equation}
If the exponential tail of $\rho_{\mathbf{x}_{0}}(E_{1}, E_{2},
s)|_{(E_{1}, E_{2})}$ has the same slope $R(E_{1}, E_{2})$ in a
log-linear plot, $\Lambda(t)$
should have a tail $e^{-t}$. It is easy to see
that $\Lambda (t)$ can be sampled by $\{ t^{i}R(E_{1}^{(i)},
E_{2}^{(i)}) \}$, where $R(E_{1}, E_{2})$ is obtained from Figure
\ref{long}. The estimator of $\Lambda (t)$ is \begin{equation}
  \label{eq3-6}
  \hat{\Lambda}(t) = \frac{1}{N} \left | \{ t^{i} R(E^{(i)}_{1},
    E^{(i)}_{2}) > t \} \right | \,,
\end{equation}
where $N$ is the
sample size of the Monte Carlo simulation.

In Figure \ref{long}, we can see that $\hat{\Lambda} (t)$ matches
$e^{-t}$ very well for $M = 2, 3, 4$. This verifies that the conditional
return time coincides with the first collision time. 

\begin{figure}[h]
\centerline{\includegraphics[width = \linewidth]{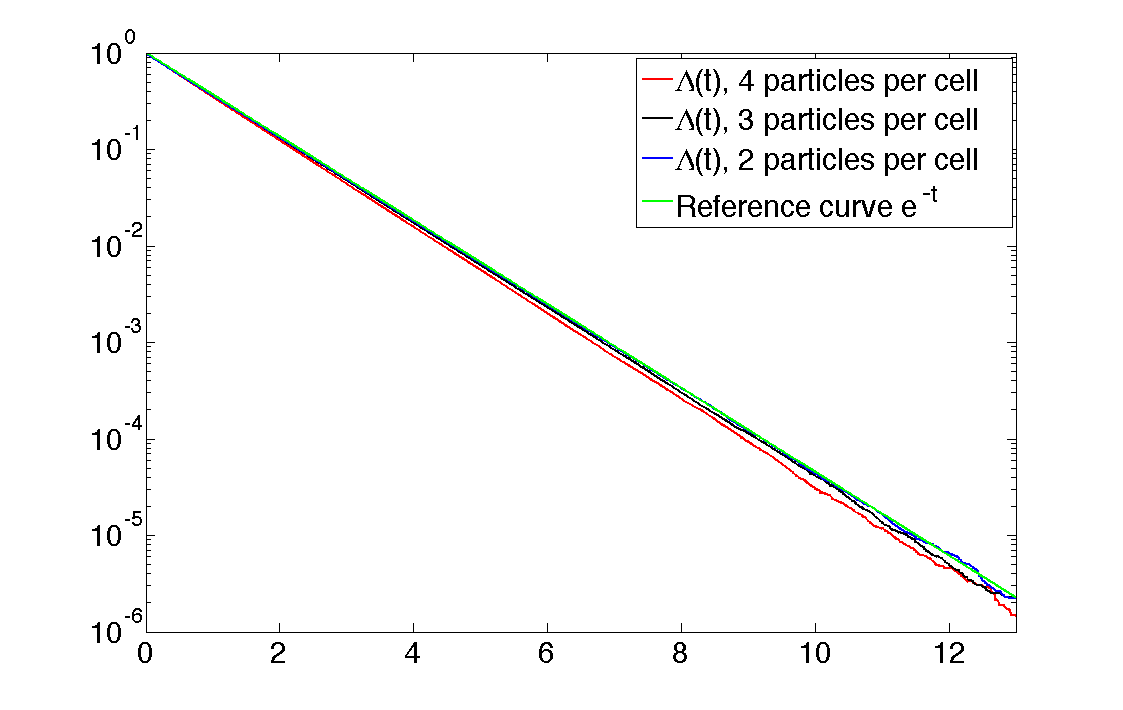}}
\caption{$\hat{\Lambda} (t)$ from long trajectories for $M = 2, 3,
  4$. $\hat{\Lambda}(t)$ is obtained from equation \eqref{eq3-6}. Sample size of each simulation is $N = 10^{8}$. }
\label{long} 
\end{figure}

\subsection{Rule of energy exchange}
The second study aims to reveal the rule of energy exchange at a
collision. We separate this problem into two parts:
\begin{itemize}
  \item[(a)] The energy distribution of a particle that participates
    in a collision. 
\item[(b)] The rule of the energy redistribution during the collision.
\end{itemize}

If there is only one particle in a cell, (a) becomes
trivial. Otherwise, consider $M$ particles collide with each other in
a chaotic billiard table. Due to the quick correlation decay, this
billiard system should converges to its invariant measure, i.e., the
Liouville measure, in a short time. At the Liouville measure, the
velocity distribution of these $M$ particles is a uniform distribution
on a $(2M - 1)$-sphere. Assume particle velocities are uniformly
distributed on a $(2M - 1)$-sphere, some easy calculation in Proposition \ref{beta}
shows that the energy distribution of a particle is a Beta distribution with
parameters $(1, M-1)$.

\begin{pro}
\label{beta}
Let $(X_{1}, \cdots, X_{2M})$ be a uniform distribution on the surface $\mathbb{S}^{2M-1}$. Then $X_{1}^{2} + X_{2}^{2}$ has Beta
distribution with parameters $(1, M-1)$.
\end{pro}
\begin{proof}

Let $Y_{1}, \cdot, Y_{2M}$ be $2M$ standard normal random
variables. Let

\begin{equation}
\label{eq3-7}
  X_{i} = \frac{Y_{i}}{ \sqrt{Y_{1}^{2} + \cdots + Y_{2M}^{2}}} \,.
\end{equation}

Then it is well known that $(X_{1}, \cdots, X_{2M})$ gives a uniform
distribution on the surface of a unit
$(2M-1)$-sphere. Therefore, we have 
\begin{equation}
  \label{eq3-8}
 X_{1}^{2} + X_{2}^{2} = \frac{ Y_{1}^{2} + Y_{2}^{2}}{ \sum_{i =
      1}^{2M} Y_{i}^{2}} := \frac{\Gamma_{1}}{\Gamma_{1} + \Gamma_{2}
  } \,,
\end{equation}
where $\Gamma_{1}$ is a $\Gamma (1, 2)$ distribution and $\Gamma_{2}$
is a $\Gamma (M-1, 2)$ distribution. The ratio $\Gamma_{1}/(\Gamma_{1}
+ \Gamma_{2})$ is a Beta distribution with parameters $(1, M-1)$. 
\end{proof}

Since each cell in the nonequilibrium billiard model forms a chaotic
billiard table, it is reasonable to assume that at the collision time,
the $M$-particle system in a cell is close to its invariant
measure. Hence the energy distribution of any given particle is
approximated by a Beta distribution with parameters $(1, M-1)$.

Some corrections need to be added to the Beta distribution to
approximate the energy distribution of the particle that {\it
  participates in } a collision between particles from neighboring cells. The reason is that faster particles have higher chance to participate in such a collision. Hence
the distribution should be biased towards high energy states. This
bias can be estimated by the following heuristic arguments. 

The billiard system in each cell is assumed to be sufficiently
chaotic, which means the correlation decays quickly. Hence it is reasonable to assume that at the collision time, particle energies in neighboring
cells are independent. Consider a pair of consecutive cells with an
energy configuration $(E_{1}, E_{2})$. Let $x_{L}$ and
$x_{R}$ be the ratio of the energy of the colliding particle to the
total energy of its cell. Because of the independence assumption, the
probability density of $x_{L}$ should be
proportional to $A(t)$, where $t$ is the ``effective time'' that a particle 
from the right table is available for a collision, and $A(t)$ is the
area swiped by a particle during the time $(0, t)$. 

We only study the energy distribution of the left particle, as the
right one follows from an analogous argument. It is obvious that
\begin{equation}
  \label{eq3-9}
A(t) = \pi R^{2} + 2 R |v| t \,,
\end{equation}
where $|v| = \sqrt{E_{1} x_{L}}$. It is not easy to give an explicit
expression of the ``effective time'', but heuristically $t$ should be
proportional to
\begin{equation}
  \label{eq3-10}
\frac{R}{\sqrt{E_{2}}}\cdot \frac{\sqrt{E_{1}}}{\sqrt{E_{2}}} \,,
\end{equation}
where the first term approximates the time duration that a particle from the
right stays at the gate area, and the second term is the ratio of ``time
scales'' in two cells. Hence we have
\begin{equation}
  \label{eq3-11}
A(t) = \pi R^{2} + 2C R^{2} \frac{E_{1}}{E_{2}} \sqrt{x_{L}} \,,
\end{equation}
where $C$ is a constant that depends on the geometry of the
model. Combine with Proposition \ref{beta}, we conclude that the
probability density of $x_{L}$ should be approximated by 
\begin{equation}
\label{beta_approx}
  w(x_{L}) = \frac{1}{K}(1 + C\frac{E_{1}}{E_{2}}\sqrt{x_{L}})(1 - x_{L})^{M - 2}
\end{equation}
if $M \geq 2$, where $K$ is a normalizer.

This heuristic argument is verified by our numerical results. In Figure
\ref{participate},
 we compare the approximation \eqref{beta_approx} with
simulation results of $x_{L}$ for three energy configurations $(E_{1}, E_{2}) =
(0.1, 0.9)$, $(0.5, 0.5)$, and $(0.9, 0.1)$. The number of particles
on each side is $4$. The constant $C$ is chosen to be $2.5$. We can
see that the approximation in equation \eqref{beta_approx} is quite close to the simulation result,
especially when $x_{L}$ is close to $1$. Note that we are more
interested in the distribution of $x_{L}$ when it is close to $1$, as
it is related to the asymptotic dynamics of the full model. 
\begin{figure}[htbp]
\centerline{\includegraphics[width = \linewidth]{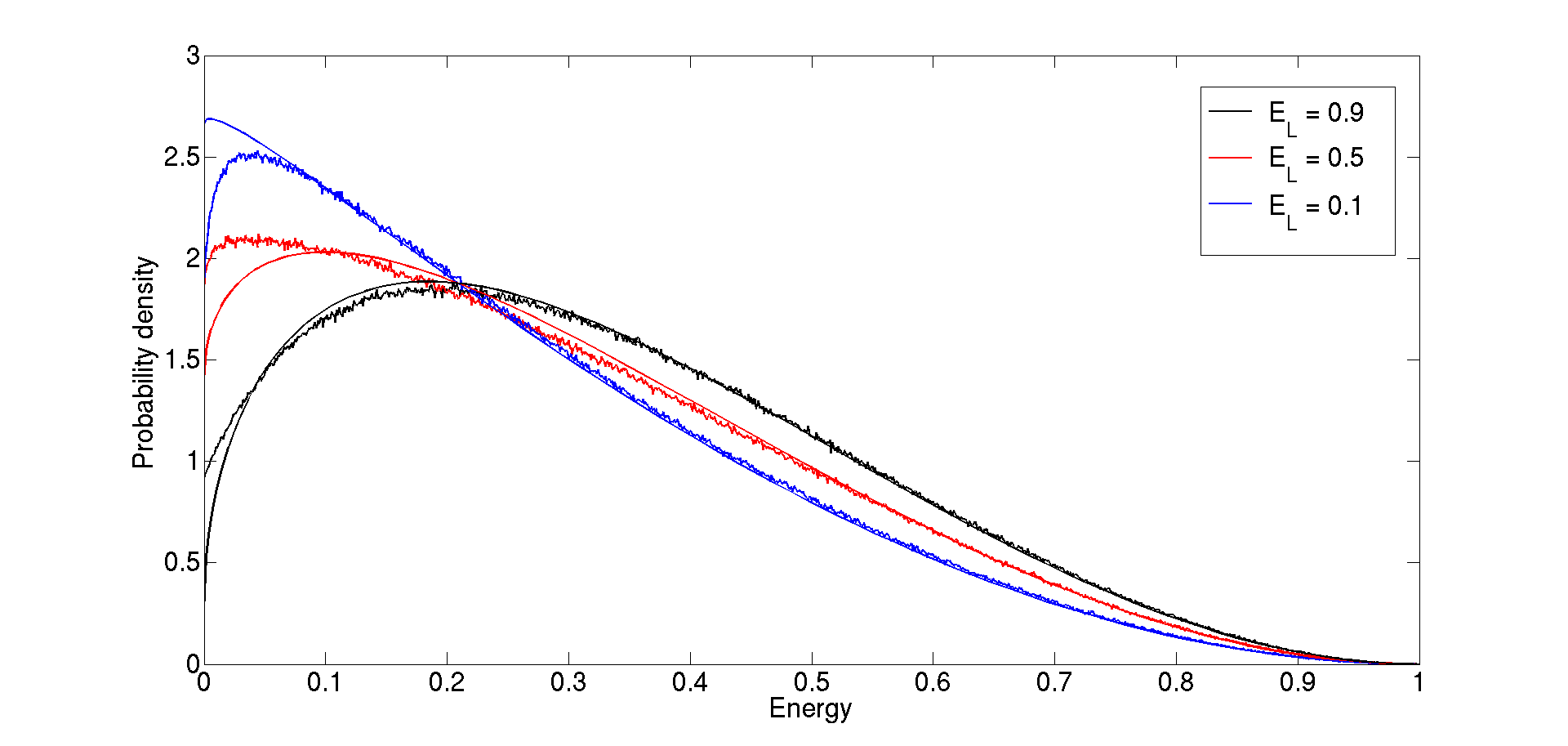}}
\caption{Energy distribution of particles participating in a collision.}
\label{participate} %beta_approx}%\label{Edist}
\end{figure}

Therefore, the energy of particles that participate in collision
should be $E_{1}B_{1}$ and $E_{2}B_{2}$, where $B_{1}$ has the
probability density function
\begin{equation}
  \label{eq3-12}
\frac{1}{K}(1 + C\frac{E_{i}}{E_{i+1}}\sqrt{x})(1 - x)^{M - 2} \,,
\end{equation}
and $B_{2}$ has the probability density function
\begin{equation}
  \label{eq3-13}
  \frac{1}{K'}(1 + C'\frac{E_{i+1}}{E_{i}}\sqrt{x})(1 - x)^{M - 2} \,,
\end{equation}
where $C, C'$ are constants and $K, K'$ are normalizers. 

To simplify the model, we only intend to capture the tail behavior of
random variables $B_{1}$ and $B_{2}$. In other words, we would like to simplify the rule of energy
exchange while preserving the right scaling when
$B_{1}$ ( or $B_{2}$) is close to $1$. The correction term in equation
\eqref{beta_approx} does not affect the tail. Hence we simplify
$B_{1}$ and $B_{2}$ to two independent random variables satisfying
Beta distributions with parameters $(1, M-1)$. This assumption is
adopted throughout the rest of this paper.  

\medskip

The answer to (b) eventually boils down to the following questions. Consider
two rigid disks moving and colliding in a ``sufficiently chaotic'' billiard table. If the
initial kinetic energies are $E_{1}$ and $E_{2}$ but the initial
position and direct of motion are both random, what will the energy distribution of each particle
be after the first collision? Without loss of generality, we assume the
velocities of two particles are
\begin{equation}
  \label{eq3-14}
\mathbf{v}_{1} = \sqrt{E_{1}}(\cos \alpha, \sin \alpha) \quad, \quad
  \mathbf{v}_{2} = \sqrt{E_{2}}(\cos \beta, \sin \beta ) \,,
\end{equation}
respectively. Let the center of mass of two particles at their first
collision be $\mathbf{x}_{1}$ and $\mathbf{x}_{2}$. Similarly we let
\begin{equation}
  \label{eq3-15}
\mathbf{x}_{1} - \mathbf{x}_{2} = 2R(\cos \gamma, \sin \gamma) \,.
\end{equation}

Assume two disks have equal mass and their mass center is the geometry
center. Then it is easy to see that the post-collision velocities
$\mathbf{v}'_{1}$ and $\mathbf{v}'_{2}$ are
\begin{equation}
  \label{eq3-16}
  \mathbf{v}_{1}' = \mathbf{v}_{1} - \frac{( \mathbf{v}_{1} -
    \mathbf{v}_{2} , \mathbf{x}_{1} - \mathbf{x}_{2}) }{ \|
    \mathbf{x}_{1} - \mathbf{x}_{2} \|^{2}}( \mathbf{x}_{1} -
  \mathbf{x}_{2} ) 
\end{equation}
and
\begin{equation}
  \label{eq3-17}
  \mathbf{v}_{2}' = \mathbf{v}_{2} - \frac{( \mathbf{v}_{2} -
    \mathbf{v}_{1} , \mathbf{x}_{2} - \mathbf{x}_{1}) }{ \|
    \mathbf{x}_{2} - \mathbf{x}_{1} \|^{2}}( \mathbf{x}_{2} -
  \mathbf{x}_{1} )  \,,
\end{equation}
respectively. Since the total energy is conservative, it is sufficient
to calculate $\| \mathbf{v}'_{1} \|^{2}$. Some calculation shows that
\begin{equation}
  \label{eq3-18}
  \| \mathbf{v}'_{1} \|^{2} = E_{1}\sin^{2}( \alpha - \gamma) +
  E_{2}\cos^{2}(\beta - \gamma) \,.
\end{equation}
Therefore, we only need to find the joint distribution of $\alpha - \gamma$
and $\beta - \gamma$. 

Let $\theta_{1} = \alpha - \gamma$ and $\theta_{2} = \beta -
\gamma$. Since a chaotic billiard system converges to its invariant
measure (Liouville measure) quickly, it is natural to assume that $\theta_{1}$ and $\theta_{2}$
are uniformly distributed on $\mathbb{S}^{1}$. Further we can assume
$\theta_{1}$ and $\theta_{2}$ to be independent at the collision time,
because two billiard systems evolve independently between two
collisions. Hence all we need is to find the conditional density of $(\theta_{1}, \theta_{2})$ when
a collision happens. Without loss of generality, we rotate the coordinate such that $\mathbf{x}_{1}
- \mathbf{x}_{2}$ is horizontal. Now assume two particles are right
before the first collision. It is easy to see that the time to
collision is proportional to $(\sqrt{E_{1}} \cos \theta_{1} -
\sqrt{E_{2}}\cos \theta_{2})^{-1}$ if $\sqrt{E_{1}} \cos \theta_{1} -
\sqrt{E_{2}}\cos \theta_{2} > 0$, and $\infty$ otherwise. In other
words, conditioning on having a collision, the approximate conditional
density of
($\theta_{1}, \theta_{2})$ should be proportional to
\begin{equation}
  \label{eq3-19}
 (\sqrt{E_{1}} \cos \theta_{1} - \sqrt{E_{2}}\cos \theta_{2} )
 \mathbf{1}_{ \{\sqrt{E_{1}} \cos \theta_{1} -
\sqrt{E_{2}}\cos \theta_{2}  > 0 \}} \,,
\end{equation}
where $\mathbf{1}_{A}$ is an indicator function with respect to set $A$. 

Therefore, the post-collision kinetic energy of particle $1$ equals 
$$
  E_{1}\sin^{2}( \theta_{1}) + E_{2}\cos^{2}(\theta_{2}) \,,
$$
where $(\theta_{1}, \theta_{2})$ has a joint probability density function
\begin{equation}
\label{eq3-20}
  \frac{1}{L}(\sqrt{E_{1}} \cos \theta_{1} - \sqrt{E_{2}}\cos \theta_{2} ) \mathbf{1}_{\{\sqrt{E_{1}} \cos \theta_{1} -
\sqrt{E_{2}}\cos \theta_{2}  > 0\}}  \,,
\end{equation}
and $L$ is a normalizer. 

This heuristic argument is justified by our numerical simulation result. In Figure
\ref{newenergy}, we demonstrate the probability density function of
the energy of the left
particle after a collision, when starting from conditional Liouville
measure conditioning on a fixed energy configuration. This matches exactly our analysis about
the post-collision energy distribution.

\begin{figure}[htbp]
\label{newenergy}
\centerline{\includegraphics[width = \linewidth]{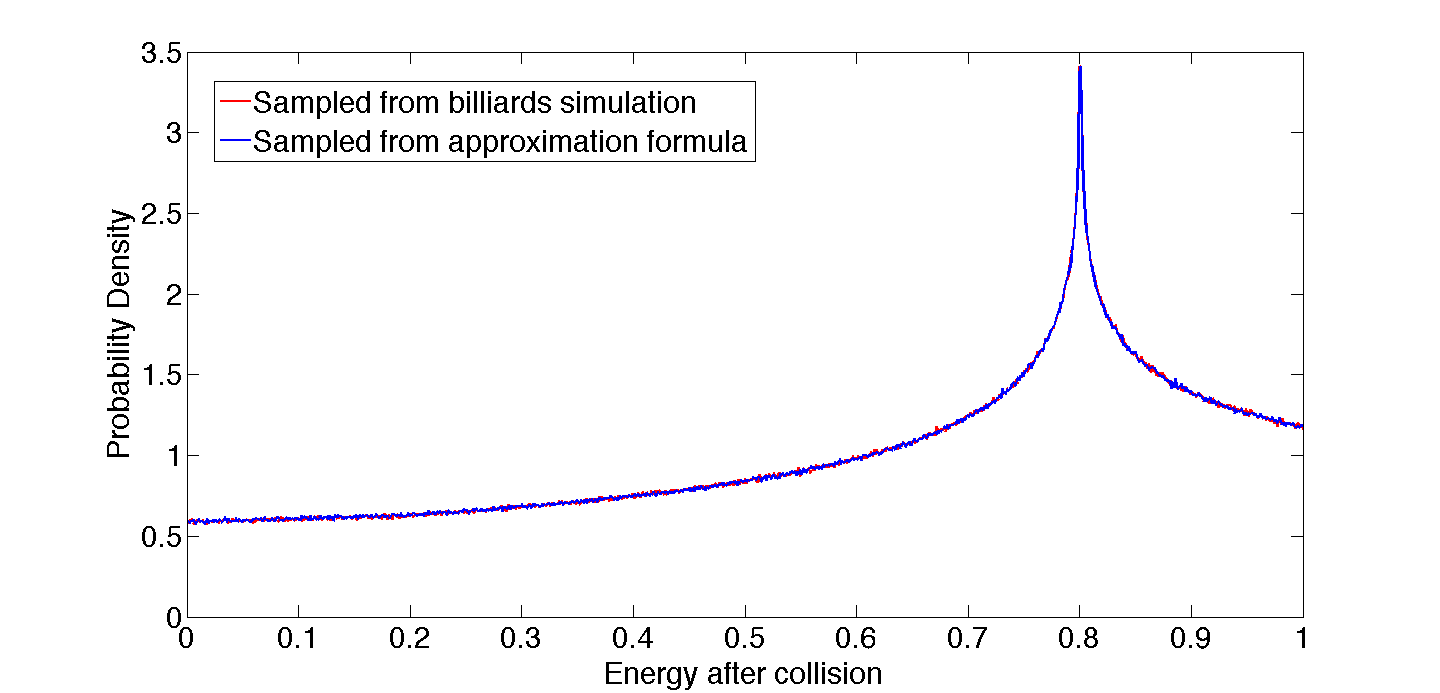}}
\caption{Energy distribution after a collision. Two particles with
  initial energy configuration $(0.8, 0.2)$ and conditional Liouville
  measure are simulated until their first collision. Red: Probability
  density function of post-collision energy of the left particle
  obtained from Monte Carlo simulation with $10^{7}$ samples. Blue: Probability
  density function of post-collision energy of the left particle given
  by equation \eqref{eq3-20}. }
\end{figure}

The joint density function in equation \eqref{eq3-20} is too
complicated to be interesting. However, it is not hard to see that for
each strictly positive energy pair $(E_{1}, E_{2})$,
the distribution of $\|\mathbf{v}'_{1}\|^{2}$ has strictly positive probability density
everywhere. Same as in (a), we look for a simple expression that
preserves the tail dynamics, which is essentially the tail probability that
$E_{1}$ (or $E_{2}$) is very small after an energy exchange. Therefore, for the sake of simplicity, we
assume that the
energy redistribution is given in a ``random halves'' fashion, i.e.,
\begin{equation}
  \label{eq3-21}
(E_{1}', E_{2}') = ( p( E_{1} + E_{2}), (1-p)(E_{1} + E_{2}) ) \,,
\end{equation}
where $p$ is uniformly distributed on $(0, 1)$. This assumption is
valid throughout the rest of this paper. 

Our numerical simulation shows that this simplification preserves the same tail
distribution as well as the same scaling of the energy
current. In Figure \ref{totalenergy}, we show the probability density
function of post-collision left cell energy when $3$ particles on each side
starting with an energy configuration $(E_{1}, E_{2}) = (0.5,
0.5)$. One can see a quadratic tail of the probability density function. This
means 
\begin{equation}
  \label{eq3-22}
\mathbb{P}[ E_{1}' < \epsilon] \propto \int_{0}^{\epsilon} s^{2}
\mathrm{d}s = O(\epsilon^{3}) \,.
\end{equation}
Finally, in Figure \ref{averageEflux} we plot the average energy flux when each side
has $3$ particles. The total energy is still set to be $1$. We can see that the energy flux is
proportional to the difference of cell energy. This further supports
the simplified rule of energy exchanges. 

\begin{figure}[htbp]
\centerline{\includegraphics[width = \linewidth]{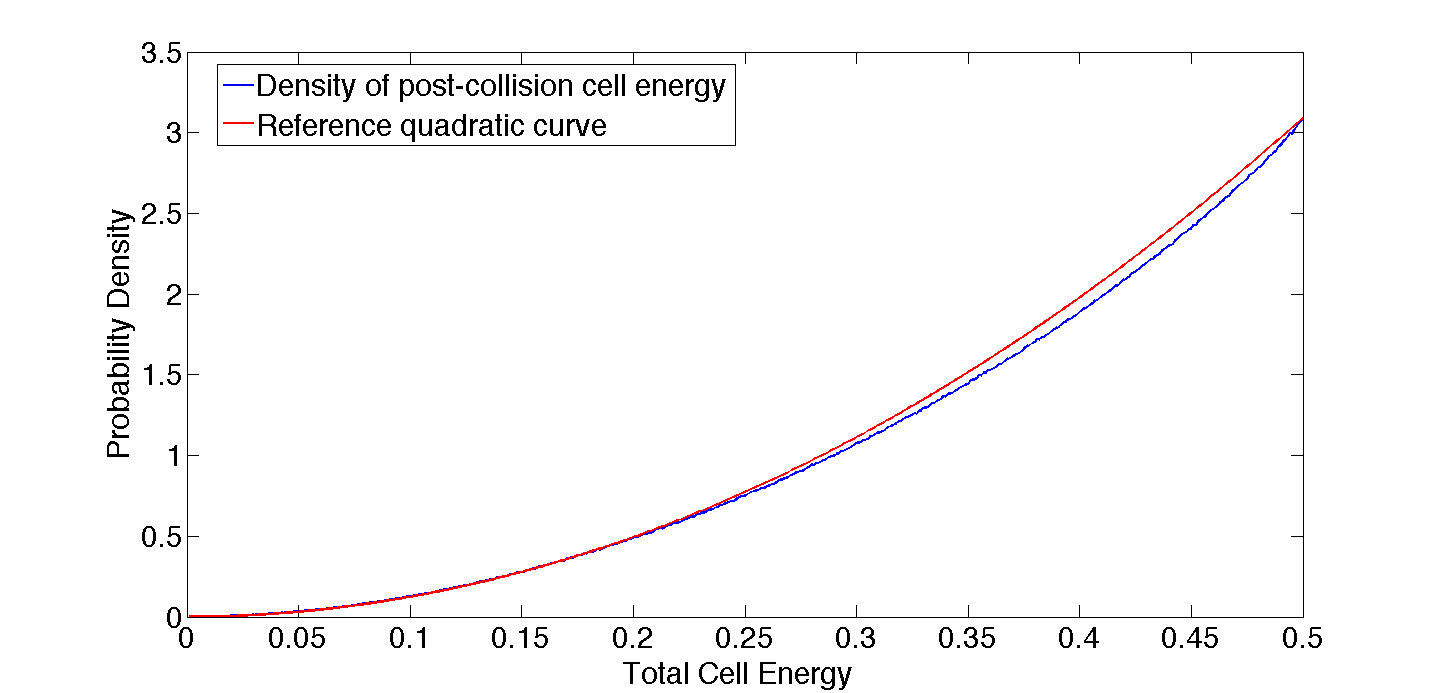}} %q, typo？
\caption{Blue: Probability density function of total left cell energy after a
  collision. Number of particles on each side $= 3$. Red: A reference
  quadratic curve. The initial distribution is the conditional
  Liouville measure conditioning on $(E_{1}, E_{2}) = (0.5,
  0.5)$. The simple size of the simulation is $10^{8}$.  }
\label{totalenergy}
\end{figure}

\begin{figure}[htbp]
\centerline{\includegraphics[width = \linewidth]{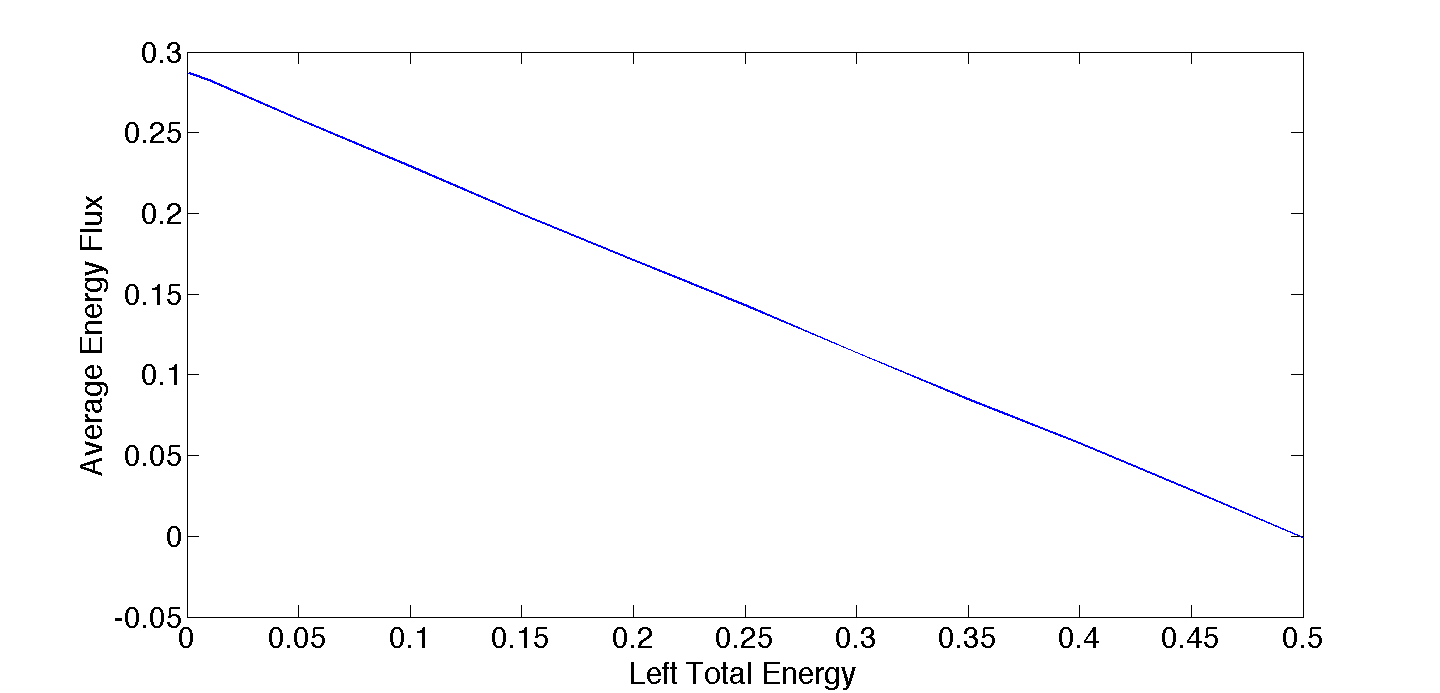}}
\caption{Average energy flux from right to left with varying energy
  configurations. Number of particles on each side $= 3$. $x$-axis is
  the total energy of the left cell $E_{1}$. Initial
  distributions are conditional Liouville measures conditioning on energy configurations $(0, 1),
  (0.05, 0.95), \cdots, (0.5, 0.5)$. Sample size for each energy
  configuration is $10^{6}$. }
\label{averageEflux}
\end{figure}

In summary, let $E_{i}$ and $E_{i+1}$ be the total local energy in two
neighboring cells, the rule of energy exchange is 

\begin{equation}
  \label{eq3-23}
(E_{i}', E_{i+1}') = ( E_{i} - (1-p)E_{i}B_{1} + pE_{i+1}B_{2},
  E_{i+1} - pE_{i+1}B_{2} + (1-p)E_{i}B_{1} ) \,,
\end{equation}
where $B_{1}$ and $B_{2}$ are two random variables with Beta distributions with parameters
$(1, M-1)$, and $p$ has uniform positive density on $(0, 1)$. Moreover, $B_{1}$,
$B_{2}$, and $p$ are independent.

\subsection{Stochastic energy exchange model}
In summary, the qualitative properties of the nonequilibrium billiard model is
preserved by the following {\it stochastic energy exchange model}. 

Consider a chain of $N$ sites that is connected to two heat
baths. Let $M$ be an integer that is the model parameter. Each site
carries a certain amount of energy. Temperatures of two heat baths are
assumed to be $T_{L}$ and $T_{R}$ respectively. As discussed in
Section 3.1, the energy exchange times can be approximated by a
Poisson distribution. Hence an exponential clock is associated to a pair of sites
$E_{i}$ and $E_{i+1}$. The rate of the clock is $R(E_{i}, E_{i+1}) =
\sqrt{\min\{ E_{i}, E_{i+1}\}}$. When the clock rings, a random proportion of
energy is chosen from each site. Then these energies are pooled
together and redistributed back randomly. The random proportion
satisfies a Beta distribution with parameters $(1, M-1)$. More precisely, the rule of
update immediately after a clock ring is as described in equation \eqref{eq3-23}.

The rule of interaction with the heat bath is analogous. Two more exponential clocks are associated to the left and the
right heat baths. The rate of the left (resp. right) clock is
$R(T_{L}, E_{1})$ (resp. $R(E_{N}, T_{R})$). When the clock ring, the
rule of update is
\begin{equation}
  \label{eq3-25}
E_{1}' = E_{1} - E_{1}B_{1} + p(E_{1}B_{1} + X_{L}B_{2}) \,,
\end{equation}
(resp. 
\begin{equation}
  \label{eq3-26}
 E_{N}' = E_{N} - E_{N}B_{1} + p(E_{N}B_{1} + X_{R}B_{2}) \,,
\end{equation}
)

where $X_{L}$ and $X_{R}$ are exponential random variables with mean
$T_{L}$ and $T_{R}$ respectively, $p$, $B_{1}$, $B_{2}$ are same as
before. 

The stochastic energy exchange model generates a Markov jump process $\mathbf{E}_{t}$
on $\mathbb{R}^{N}_{+}$. We denote $P^{t}$ by the transition kernel of
$\mathbf{E}_{t}$. We further define the left
operator of $P^{t}$ acting on a probability measure $\mu$
\begin{equation}
  \label{eq3-27}
\mu P^{t}(A) = \int_{\mathbb{R}^{N}_{+}} \mu(\mathrm{d}x) P^{t}(x, A)
\end{equation}
and the right operator of $P^{t}$ acting on a measurable function
$\xi$
\begin{equation}
  \label{eq3-28}
P^{t}\xi(x) = \int_{\mathbb{R}^{N}_{+}} P^{t}(x, \mathrm{d}y)\xi(y) \,.
\end{equation}

\section{Comparison of ergodicity of deterministic and stochastic models}
The reduction from billiard model to the stochastic energy model aims
to preserve the long time dynamics. In this section, we will use
numerical and analytical tools to verify that asymptotic dynamics are
preserved and that the two models have similar ergodicity. For
ergodicity, we mean the existence and uniqueness of nonequilibrium
steady state, the speed of
convergence to steady state, and the rate of correlation decay. Section
\ref{sec3-1} gives a computer assisted proof of ergodicity for the
stochastic model. Besides some numerical estimates regarding return times, all
arguments are rigorous. In comparison, proving the ergodicity of the nonequilibrium
billiard model is much more difficult. Instead, we provide
some numerical evidence together with heuristic arguments to justify
that the billiard model has the same rate of correlation decay.

\subsection{Probability preliminary on ergodicity of Markov processes}
\label{sec3-1}

A Markov process admits a unique ergodic invariant probability measure
under some drift conditions \cite{meyn2012markov, hairer2010convergence}. There are also existing results for
the speed of convergence to its invariant measure and the rate of
correlation decay \cite{meyn2012markov, hairer2010convergence}. However, these known results can not be applied to
the stochastic energy exchange model directly. Even proving the
simplest case ($M=1$) requires advanced techniques
 and very tedious calculations \cite{li2017polynomial}. It is very difficult to show the speed of convergence through a direct
Monte Carlo simulation either. The decay of correlation
has small expectation but $O(1)$ variance. To reduce the relative
error, a huge amount of samples will be necessary. If the speed of
convergence is slow, such a simulation becomes impractical. 

Instead, in this subsection we introduce a hybrid approach proposed
in \cite{li2017numerical}. This method circumvents main difficulties
of both analytical proof and direct Monte Carlo simulations. It gives
an easy and convincing justification of the ergodicity of a Markov
process on any measurable state space. Below we will focus on this
hybrid method for time continuous Markov processes.

Let $\Psi_{t}$ be a continuous time Markov process on a measure state
space $(X, \mathcal{B})$. Let $h > 0$ be a fixed constant. Denote the
time-$h$ sample chain of $\Psi_{t}$ by $\Psi^{h}_{n}$, i.e.,
$\Psi^{h}_{n} = \Psi_{nh}$. Let $\mathcal{P}(x, \cdot)$
be the transition kernel of $\Psi^{h}_{n}$. Further we define
$\tau_{A}(h) = \inf_{t \geq h}\{ \Psi_{t} \in A \}$.  

The theory of Markov processes on measurable state spaces is quite
different from that of Markov chains on countable spaces. We refer
\cite{meyn2012markov} for a detailed review of this subject. Below we
only introduce some necessary terminologies to use the hybrid method
in \cite{li2017numerical}.

Let $\phi$ be a measure on $(X, \mathcal{B})$. $\Psi^{h}_{n}$ is said
to be {\it $\phi$-irreducible} if for any $x
    \in X$ and any $A \in \mathcal{B}$ with $\phi(A) > 0$, there
    exists an integer $n > 0$ such that $\mathcal{P}^{n}(x, A) > 0$. 

A measurable set $\mathfrak{C} \subset X$ is said to be a {\it uniform
reference set} if 
\begin{equation}
  \label{eq4-1}
  P^{h}(x, \cdot) \geq \eta \theta(\cdot) \quad \mbox{ for all } x \in
  \mathfrak{C} \,,
\end{equation}
where $\theta( \cdot)$ is a nontrivial probability measure. 

The Markov chain $\Psi^{h}_{n}$ is said to be {\it strongly aperiodic}
if it admits a uniform reference set that satisfies $\theta
(\mathfrak{C}) > 0$. 

Finally, $\Psi_{t}$ is said to satisfy the ``continuity at zero''
condition if for any probability measure $\mu$, we have $
  \| \mu P^{\delta} - \mu \|_{TV} \rightarrow 0 \mbox{ as } \delta
  \rightarrow 0 $,  where $\| \cdot \|_{TV}$ is the total variation norm. 

By \cite{li2017numerical}, in order to show the polynomial ergodicity
of $\Psi_{t}$, we need the following four analytical conditions and two numerical conditions.
\begin{itemize}
  \item[{\bf (A1)}] $\Psi^{h}_{n}$ is irreducible with respect to a
    non-trivial measure $\phi$.
\item[{\bf (A2)}] $\Psi^{h}_{n}$ admits a uniform reference set
  $\mathfrak{C}$ and is strongly aperiodic. 
\item[{\bf (A3)}] $\Psi_{t}$ satisfies the ``continuous at zero''
  condition.
\item[{\bf (A4)}] There exists $\gamma > 0 $ such that
\begin{equation}
  \label{eq4-2}
\inf_{x \in \mathfrak{C}}  \inf_{t\in [0, h]}  P_{x}[ \Psi_{t} = \Psi_{0}] > \gamma \,.
\end{equation}
\item[{\bf (N1)}] Distributions $\mathbb{P}_{\mu}[ \tau_{\mathfrak{C}}(h)
  \geq t]$ and $\mathbb{P}_{\pi}[ \tau_{\mathfrak{C}}(h) \geq t]$ have
  polynomial tails $\sim t^{-\beta}$ for some $\beta > 1$, where $\pi$
  is the numerical invariant measure of $\Phi_{t}$.
\item[{\bf (N2)}] Function
\begin{equation}
  \label{eq4-3}
  \gamma(x) = \sup_{t \geq h} \frac{\mathbb{P}_{x}[
    \tau_{\mathfrak{C}}(h) > t]}{t^{-\beta}}
\end{equation}
is uniformly bounded on $\mathfrak{C}$.
\end{itemize}

In \cite{li2017numerical}, we have showed that conditions {\bf (A1) -- (A4)},
{\bf (N1)}, and {\bf (N2)} implies the following conclusions. 

\begin{itemize} 
  \item[(a)] $\Psi_{t}$ admits an invariant probability measure
    $\pi$.
\item[(b)] Polynomial convergence rate to $\pi$:
\begin{equation}
  \label{eq4-4}
\lim_{t\rightarrow \infty}  t^{\beta - \epsilon} \| \mu \mathcal{P}^{t} - \pi
\|_{TV}  = 0 
\end{equation}
for any $\epsilon > 0$.
\item[(c)] Polynomial decay rate of correlation:
\begin{equation}
  \label{eq4-5}
\lim_{t\rightarrow \infty}  t^{\beta - \epsilon} C_{\mu}^{\xi, \eta} = 0 
\end{equation}
for any $\epsilon > 0$ and probability measure $\mu$ satisfying {\bf
  (N1)}, where
\begin{equation}
  \label{eq4-6}
C^{\xi, \eta}_{\mu}(t) := |\int (\mathcal{P}^{t} \eta)( x) \xi(x)
  \mu(\mathrm{d}x) - \int (\mathcal{P}^{t} \eta)( x) \mu(\mathrm{d}x) \int
\xi(x) \mu( \mathrm{d}x) | \,.
\end{equation}
  \item[(d)] Polynomial convergence rate to $\pi$. For any $\epsilon >
    0$, we have
\begin{equation}
  \label{eq4-7}
\lim_{t\rightarrow \infty}  t^{\beta - \epsilon} \| \delta_{x} \mathcal{P}^{t} - \pi \|_{TV}  = 0
\end{equation}
for $\phi$-almost every $x \in X$. 
\item[(e)] Polynomial speed of contraction. For any $\epsilon >
    0$, we have
\begin{equation}
  \label{eq4-8}
\lim_{t\rightarrow \infty}  t^{\beta - \epsilon} \| \delta_{x}
  \mathcal{P}^{t} - \delta_{y} P^{t}\|_{TV}  = 0
\end{equation}
for $\phi$-almost every $x, y \in X$. 
\end{itemize}

Note that we did not specify conclusion (e) in \cite{li2017numerical}. But (e)
is a natural corollary of Proposition 4.1 of \cite{li2017numerical}, which
implies $\mathbf{E}_{x}[ \tau^{\beta}_{\mathfrak{C}}] < \infty$ for
$\phi$-almost $x \in X$.

\subsection{Verifying analytical conditions} We will first work on the
time-$h$ chain $\mathbf{E}_{n}$. The verification of condition {\bf
  (A1)} for $\mathbf{E}_{n}$ is based on the following Theorem.

\begin{thm}
\label{urs1}
For any set $K \subset \mathbb{R}^{N}_{+}$ of the form $K = \{(e_{1},
\cdots, e_{N} ) \,|\, 0 < c_{i} \leq e_{i} \leq C_{i} , i = 1 \sim N \}$ and any $h > 0$,
there exists a constant $\eta > 0$ such that
\begin{equation}
  \label{eq4-9}
P(\mathbf{E}, \cdot) > \eta U_{K}(\cdot) \,,
\end{equation}
for any $\mathbf{E} \in K$, where $U_{K}$ is the probability measure
for the uniform distribution on $K$.
\end{thm}
\begin{proof}
This proof is similar to Theorem 5.1 of \cite{li2017numerical}. We include the
proof here for the completeness of the paper. Consider any point
$\mathbf{E}^{*}=\{e^{*}_1,\ldots, e^{*}_N\} \in K$ and any small
vector $\mathrm{d}\mathbf{E}=\{ ( \mathrm{d}e_{1}, \cdots, \mathrm{d}e_{N}), de_i >0,
i=1 \sim N\}$. Assume $0 < \mathrm{d}e_{i} \ll 1$ and let
\begin{equation}
  \label{eq4-10}
B(\mathbf{E}^{*},\mathrm{d}\mathbf{E}) = \{ (x_{1}, \cdots, x_{N}) \in \mathbb{R}^{N} \, | \, e^{*}_i \leq x_i \leq
e^{*}_i+\mathrm{d}e^{*}_i \}
\end{equation}
be a small hypercube close to  $\mathbf{E}^*$. It then suffices to prove
that for any $\mathbf{E}_{0} = \{ \bar{e}_{1}, \cdots, \bar{e}_{N}
\} \in K$, we have
\begin{equation}
\label{eqn:analytic}
P(\mathbf{E}_{0}, B( \mathbf{E}^*,d\mathbf{E})) > \sigma \mathrm{d}e_{1}\mathrm{d}e_{2}\cdots
\mathrm{d}e_{N} \,,
\end{equation}
where $\sigma$ is a strictly positive constant that only depends on
$K$.

We then construct the following sequence of events to go from the state
$\mathbf{E}_{0}$ to $B(\mathbf{E}^{*}, \mathrm{d}\mathbf{E})$ with desired positive
probability. Denote the process starting from $\mathbf{E}_{0}$ by
$\mathbf{E}_{t} = (e_{1}(t), \cdots, e_{N}(t))$. Let $\delta = \frac{h}{2N+1}$ and let $\epsilon>0$ be
sufficiently small such that $\epsilon < \min \{c_i,i=1 \sim N\}$. Let
$H=\sum_{i=1} ^N (e^{*}_i+\mathrm{d}e_i)$. We consider events $A_1 \cdots, A_N$
and $B_1, \ldots, B_{N+1}$, where $A_i$ and $B_j$ specifies what
happens on the time interval $(i\delta, (i+1)\delta]$ and
$( N\delta + (j-1)\delta, N \delta + j\delta]$, respectively.

\begin{itemize}
  \item $A_{i} =$ $\{e_i(i\delta) \in [\epsilon/2,\epsilon]\}$ and  \{
    The
    $i$-th clock rings exactly once, all other clocks are silent on
    $((i-1)\delta, i\delta]$ \}. 
    \item $B_{1} = $ {Energy emitted by right heat bath $\in (H, 2H)$
      } and { the
    $N$-th clock rings exactly once, all other clocks are silent on $(
    N\delta, (N+1) \delta]$ }.
\item $B_{j} = $ $\{e_j(N\delta + j \delta) \in [e^{*}_{N+2 - j},
  e^{*}_{N+2-j}+\mathrm{d}e_j]\}$ and $\{$ the
    $(N+1-j)$-th clock rings exactly once, all other clocks are silent on $(
    N\delta + (j-1) \delta, N\delta + j \delta]$ $\}$ for $j = 2, \cdots,
  N+1$. 
\end{itemize}

The idea is that the energy at each site is first transported to the right
heat bath, with only an amount of energy between $\epsilon/2$ and
$\epsilon$ left at each site (events $A_{1} \sim A_{N}$). Then a sufficiently large amount of energy is injected into
the chain from the right heat bath (event $B_{1}$) so that it is always possible for
site $j$ to acquire an amount of energy between $e^{*}_j$ and $e^{*}_j+\mathrm{d}e_j$
by passing the rest to site $j-1$ (events $B_{2} \sim B_{N+1}$), where sites $0$ and $N+1$ denote
the left and right heat baths respectively.  

It is easy to show that for each parameter $M$, the probability of occurrence of the sequence of
events described above is always strictly positive. Below is a sketch
of calculation. We leave detailed calculations to the reader.   
\begin{enumerate}
\item[(a)] After each energy exchange, the rate of clocks have a
  uniform lower bound $\epsilon/2$. 
\item[(b)] By the rule of energy
  redistribution, it is easy to see that the probabilities of $A_{i}$ are strictly positive.
\item[(c)] There is also a uniform upper bound on $H$ given by $2\sum_{i=1} ^N C_i$. 
\item[(d)] From the rule of energy redistribution, the probability that $e_j(N\delta + j \delta)
  \in (e^{*}_j, e^{*}_j+\mathrm{d}e_j)$ after an energy exchange in
  event $B_{j+1}$ is greater than $\alpha \mathrm{d}e_{j}$ for some strictly positive
  constant $\alpha$. Hence probabilities
  of $B_j$ are greater than $\mathrm{const} \cdot \mathrm{d}e_{j}$. 
\end{enumerate}
In addition, all these probabilities are uniformly bounded from below
for all $\mathbf{E}^{*}$ and $\mathbf{E}_{0}$ in $K$. Hence we have
\begin{equation}
  \label{eq4-11}
\mathbb{P}[ A_{1} \cdots A_{N}B_{1} \cdots B_{N+1}] \geq \sigma
  \mathrm{d}e_{1}\cdots \mathrm{d}e_{N}
\end{equation}
for some constant $\sigma > 0$.
\end{proof}

As a corollary, we can prove that $\mathbf{E}_{n}$ is both strongly aperiodic and irreducible 
with respect to the Lebesgue measure.

\begin{cor}
\label{aperiod}
$\mathbf{E}_{n}$ is a strongly aperiodic Markov chain.
\end{cor}
\begin{proof}
By theorem \ref{urs1}, $K$ is a uniform reference set. In addition $U_{K}(K) > 0$. The strong aperiodicity follows from
its definition.
\end{proof}

Therefore $\mathbf{E}_{n}$ is strongly aperiodic. 

\begin{cor}
\label{irreducible}
$\mathbf{E}_{n}$ is $\lambda$-irreducible, where $\lambda$ is the
Lebesgue measure on $\mathbb{R}^{N}_{+}$. 
\end{cor}
\begin{proof}
Let $A \subset \mathbb{R}^{N}_{+}$ be a set with strictly positive
Lebesgue measure. Then there exists a set $K$ that has the form $\{(e_{1},
\cdots, e_{N}) \,|\, 0 < c_{i} \leq e_{i} \leq C_{i} , i = 1 \sim N
\}$ and $U_K(K \cap A) > 0$. 

For any $\mathbf{E}_{0} \in \mathbb{R}^{N}_{+}$ and the time
step $h > 0$, we can choose a $K \subset \mathbb{R}^{N}_{+}$ of the form $K = \{(e_{1},
\cdots, e_{N} ) \,|\, 0 < c_{i} \leq e_{i} \leq C_{i} , i = 1 \sim N
\}$ for some $c_{i} > 0$ and $C_{i} < \infty$, such that
$\mathbf{E}_{0} \in K$. Same construction as in Theorem \ref{urs1} implies that $P^{h}(\mathbf{E}_{0},
\cdot) > \eta U_{K}( \cdot)$ for some $\eta > 0$. Therefore, $P^{h}(
\mathbf{E}_{0}, A) > \eta U_{K}(A)  > 0$. 

\end{proof}

Hence assumption {\bf (A1)} and {\bf (A2)} are satisfied.

We can also prove the absolute continuity of $\pi$ with
respect to the Lebesgue measure, which is denoted by $\lambda$. 

\begin{pro}
\label{abscont}
If $\pi$ is an invariant measure of $\mathbf{E}_{t}$, then $\pi$ is
absolutely continuous with respect to $\lambda$ with a strictly positive density. 
\end{pro}
\begin{proof}
This proof is identical to that of Lemma 6.3 of
\cite{2014LiYoungNonlinearity}. 
\end{proof}

Condition {\bf (A3)}, or ``continuity at zero'' follows from the
following Proposition.
\begin{pro}
\label{cont0}
For any probability measure $\mu$ on $\mathbb{R}^{N}_{+}$, $\lim_{\delta \rightarrow 0} \| \mu P^{\delta } - \mu \|_{TV} = 0$.
\end{pro}
\begin{proof}
This proof is identical to that of Lemma 5.6 of \cite{li2013existence}.
\end{proof}

Condition {\bf (A4)} is trivial as all clock rates are uniformly
bounded in any compact set $\mathfrak{C}$.

\subsection{Verifying numerical conditions}

Now we are ready to present our numerical results. The demonstrated
results are for $N = 3$ and $M = 2$, while our conclusion holds for
other parameters we have tested. The uniform
reference set $\mathfrak{C}$ is chosen as 
\begin{equation}
  \label{eq4-12}
  \mathfrak{C} = \{ (e_{1}, \cdots, e_{N}) \,|\, 0.1 \leq e_{i} \leq
  100 , i = 1 \sim N\} \,.
\end{equation}

Throughout our numerical justification, we let $h = 0.1$. (Recall that
for a time-continuous Markov process $\Psi_{t}$, the definition of
$\tau_{\mathfrak{C}} = \tau_{\mathfrak{C}}(h)$ depends on $h$.) Our
numerical simulation shows that the tail of $\mathbb{P}_{\mathbf{E}}[
\tau_{\mathfrak{C}} > t]$ is $\sim t^{-4}$ for many initial condition
$\mathbf{E}$ that we have tested. This is consistent with the
heuristic argument. The tail of $\mathbb{P}_{\pi}[\tau_{\mathfrak{C}}
> t ]$ is a very subtle issue as an explicit formulation of $\pi$ is
not possible. We conjecture that $\mathbb{P}_{\pi}[\tau_{\mathfrak{C}}
> t ] \sim t^{-3}$.

We have the following argument and numerical evidence to support this conjecture. Consider the simplest case when $N = 1$. If $\pi( \{ E_{1}
< \epsilon \})$ have the tail $\epsilon^{p}$ for all sufficiently small
$\epsilon$, then the probability density function at $E_{1} = \epsilon$ is $\sim \epsilon^{p-1}$. Since $\pi$ is invariant, for an infinitesimal $h >
0$, we have
 \begin{equation}
  \label{eq4-13}
O(h)\epsilon^{M}  \approx O(h)\int_{0}^{\epsilon} s^{p-1} \sqrt{s} \mathrm{d}s \,,
\end{equation}
where the left term is the probability that $E_{1} < \epsilon$ after
one energy exchange within $(0, h)$, and the right term is the probability that
$E_{1}$ exchanges energy with $(0, h)$. This implies $p = M - 1/2$.

One needs to be very careful about the initial distribution when
computing the numerical invariant probability measure, as it takes a long time for the model to
converge to the steady state. As shown below, the slow convergence
mainly occurs at low energy sets. Our strategy is to generate a
numerical invariant probability measure from a initial distribution
with a correct tail. Let $\mu_{0} \sim (\rho_{1} , \cdots, \rho_{N} )$, where $\rho_{i}$ is an exponential
distribution with mean $(T_{L } + T_{R})/2$. We manually correct the
tail of $\mu_{0}$ before putting it into the Monte Carlo
simulation. This manual correction gives a new initial distribution $\mu_{1} \sim (\rho_{1} , \cdots, \rho_{N} )$, where
\begin{equation}
\label{eq4-14}
  \rho_{i} \sim \left \{ 
\begin{array}{ll}
 \mathcal{E}( (T_{L } + T_{R})/2 )& \mbox{ if }  \mathcal{E}( (T_{L } +
                                                        T_{R})/2 ) >
                                                        0.01 \\
0.01 u^{(M - 1/2)^{-1}} & \mbox{otherwise}
\end{array}
\right . \,,
\end{equation}
and $\mathcal{E}(\lambda)$ means an exponential random variable with mean
$\lambda$. 

We use the following simulation to justify this correction. The expectation of $E_{2}$ versus time is plotted in Figure
\ref{sstest}, which is stabilized quickly. In fact, expectations of
most observables we have tested converge very fast. However, the slow
convergence phenomenon can be captured at the tail, as seen in Figure
\ref{comtail}. The tail of $\mu_{0}P^{200}$ and
$\mu_{1}P^{200}$ are compared in Figure \ref{comtail}, in which we
find that the low energy tail of $\mu_{0}$ has not been 
stabilized yet. This problem is solved by using $\mu_{1}$. This
prompts us to choose $\hat{\pi} = \mu_{1}P^{100}$ as the numerical invariant measure. Our simulation shows that $\mathbb{P}_{\hat{\pi}}[\tau_{\mathfrak{C}} > t] \sim t^{-3}$. 
(See Figure \ref{SEE_ss}. )

\begin{figure}[h]
\centerline{\includegraphics[width = \linewidth]{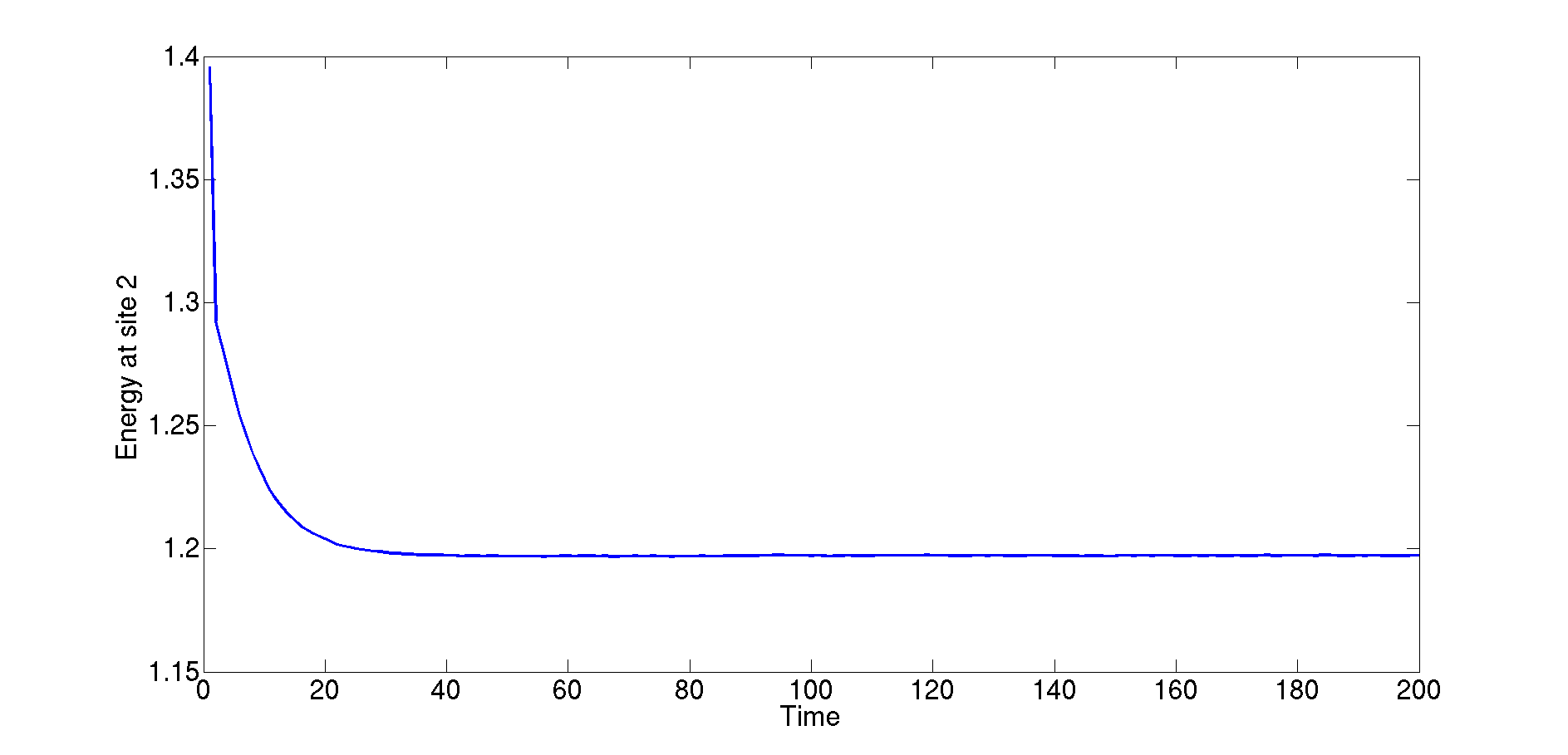}}
\caption{Expectation of $E_{2}$ vs. time in the stochastic energy
  exchange model. Model parameters are $T_{L} = 1, T_{R} = 2$, $N = 3$, and $M =
  2$. Sample size of Monte Carlo simulation is $10^{9}$.} 
\label{sstest}
\end{figure}

\begin{figure}[h]
\centerline{\includegraphics[width = \linewidth]{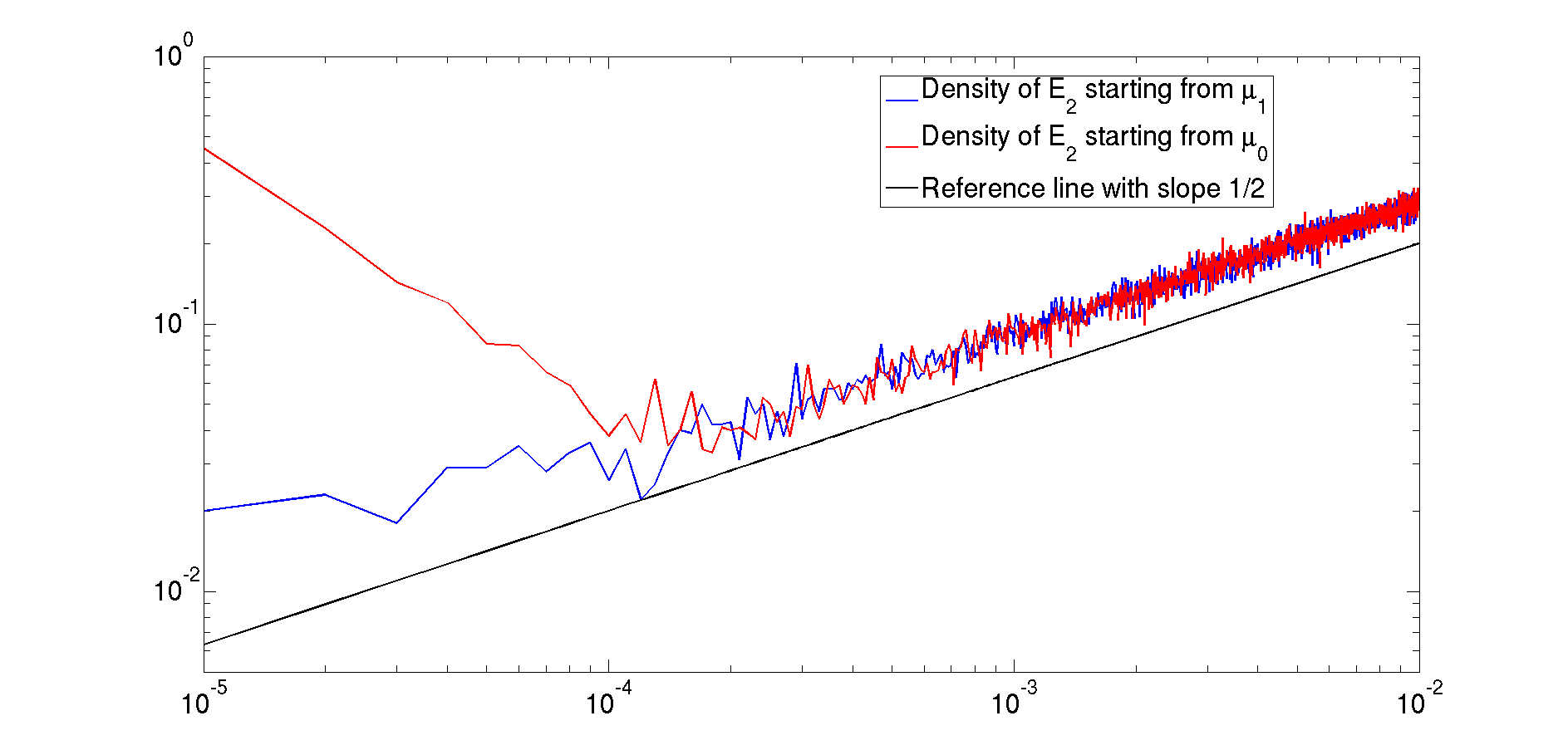}}
\caption{Probability density functions of $E_{2}$ at $T = 200$ when starting from initial
  distributions $\mu_{0}$ and $\mu_{1}$ in a log-log plot. Model parameters are $T_{L} = 1, T_{R} = 2$, $N = 3$, and $M =
  2$. The interval $[0, 0.01]$ are divided into
  $1000$ bins. The probability density is estimated by counting samples
  whose $E_{2}$ falls into each bin at $T = 200$. Sample size of Monte Carlo
  simulation is $10^{9}$.} 
\label{comtail}
\end{figure}

\begin{figure}[h]
\centerline{\includegraphics[width = \linewidth]{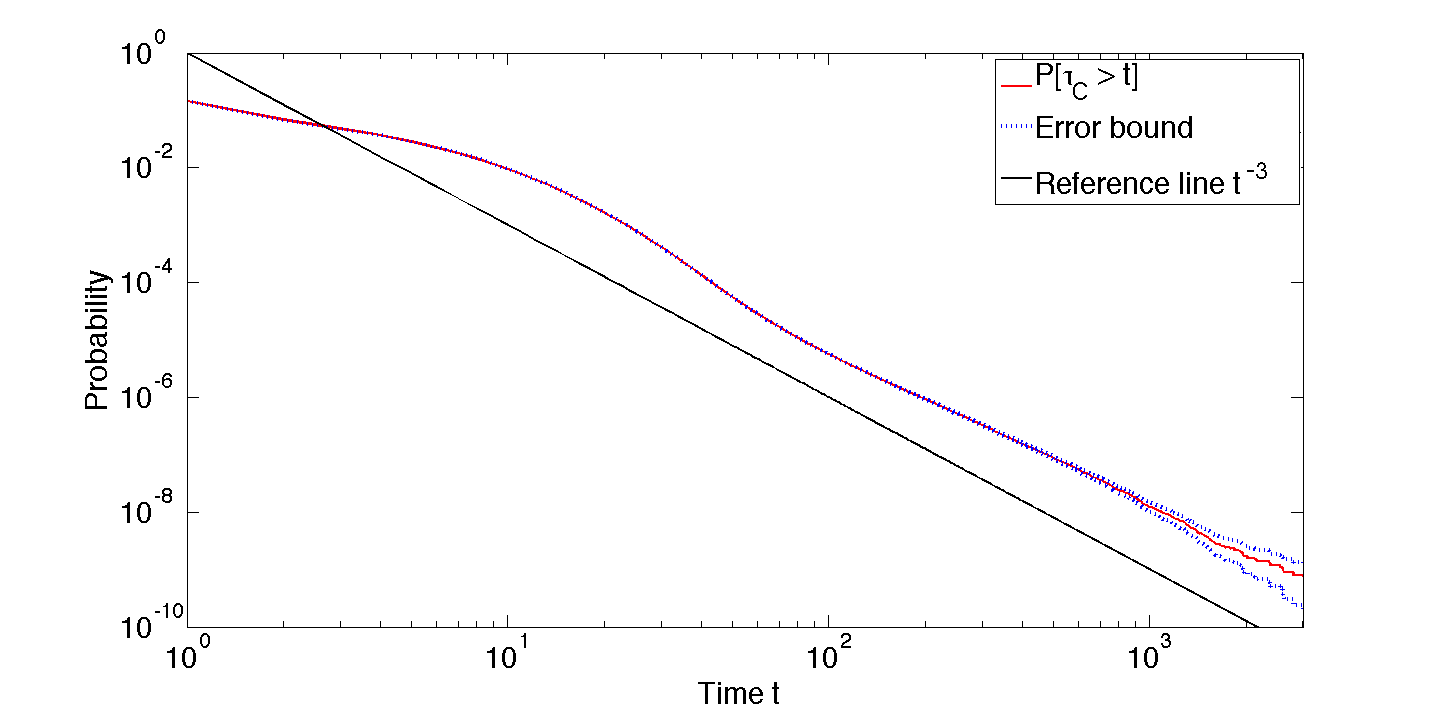}}
\caption{$\mathbb{P}[ \tau_{\mathfrak{C}} > t]$ versus $t$ when
  starting from $\hat{\pi}$ in a log-log plot. Model parameters are $T_{L} = 1, T_{R} =
  2$, $N = 3$, and $M = 2$. Blue dots
  are the error bar with confidence level 0.95. Black line is a
  reference line with slope $-3$. Sample size of Monte
  Carlo simulation is $10^{10}$.} 
\label{SEE_ss}
\end{figure}

It remains to check {\bf (N2)}. We numerically show that 
\begin{equation}
  \label{eq4-15}
 \gamma (\mathbf{E}) = \sup_{t \geq h}\frac{\mathbb{P}_{\mathbf{E}}[\tau_{\mathfrak{C}}
    > t]}{t^{-4}}
\end{equation}
is uniformly bounded on $\mathfrak{C}$. We follow procedure (a)-(d) in
Section 4.1 to show the boundedness of $\gamma (\mathbf{E})$. In fact, 
\begin{equation}
  \label{eq4-16}
  \gamma_{N}(\mathbf{E}) = \sup_{1 \leq n \leq N} \sup_{t \geq h}\frac{\mathbb{P}_{\mathbf{E}}[\tau_{\mathfrak{C}}
    > t]}{t^{-2}} 
\end{equation}
is stabilized very fast with increasing $N$. We find that a sample of size $10^6$ is sufficient for a reliable
estimate of $\gamma (\mathbf{E})$. Figure \ref{SEE_search} shows that
when $E_{i}$ is small, $\gamma (\mathbf{E})$ decreases monotonically
with decreasing $E_{i}$ for each $i = 1 \sim 3$. Therefore, we expect that the maximal of $\gamma(\mathbf{E})$ in $\mathfrak{C}$
is reached at $\mathbf{E}_{*} = (0.1, 0.1, 0.1)$. In fact, intuitively one should expect
$\gamma(\mathbf{E})$ to decrease with site energy, as starting from low site energy
means having higher probability to have even lower site energy after an 
energy exchange. Finally, we run the simulation again to estimate
$\mathbb{P}_{\mathbf{E}^{*}}[\tau_{c} > t]$. As seen in Figure \ref{SEE_max}, when
starting from $\mathbf{E}_{*}$, $\mathbb{P}_{\mathbf{E}^{*}}[\tau_{c}
> t]$ has a tail
$\sim t^{-4}$.

\begin{figure}[h]
\centerline{\includegraphics[width = \linewidth]{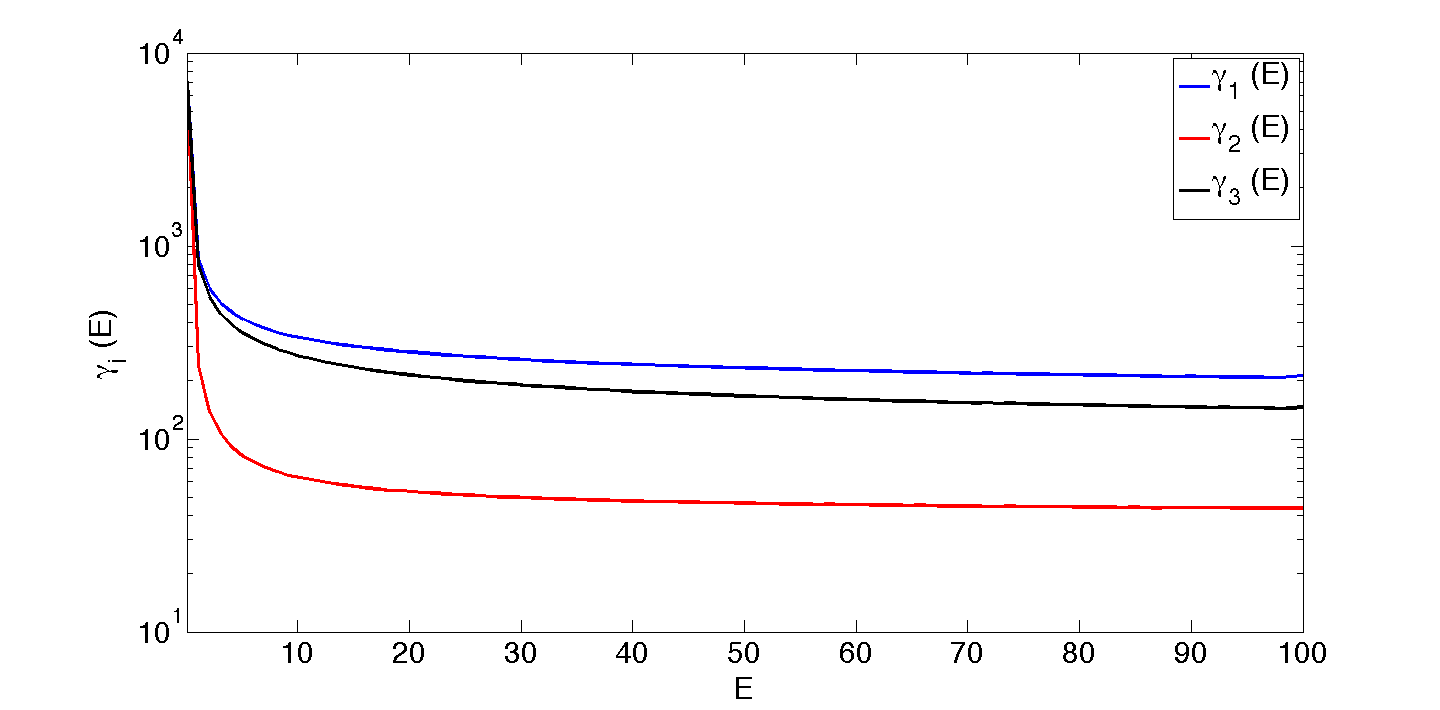}}
\caption{The label $\gamma_{i}(E)$ means replacing the $i$-th entry of
  $\gamma (0.1, 0.1, 0.1)$ by $E$. Model parameters are $T_{L} = 1, T_{R} = 2$, $N = 3$, and $M =
  2$. Three curves plot
  $\gamma_{i}(E)$ on $[0.1, 100]$ for $i = 1 \sim 3$. Linear-log plot
  is used because values of
  $\gamma_{i}(E)$ changes significantly when $E$ is small. } 
\label{SEE_search}
\end{figure}

\begin{figure}[h]
\centerline{\includegraphics[width = \linewidth]{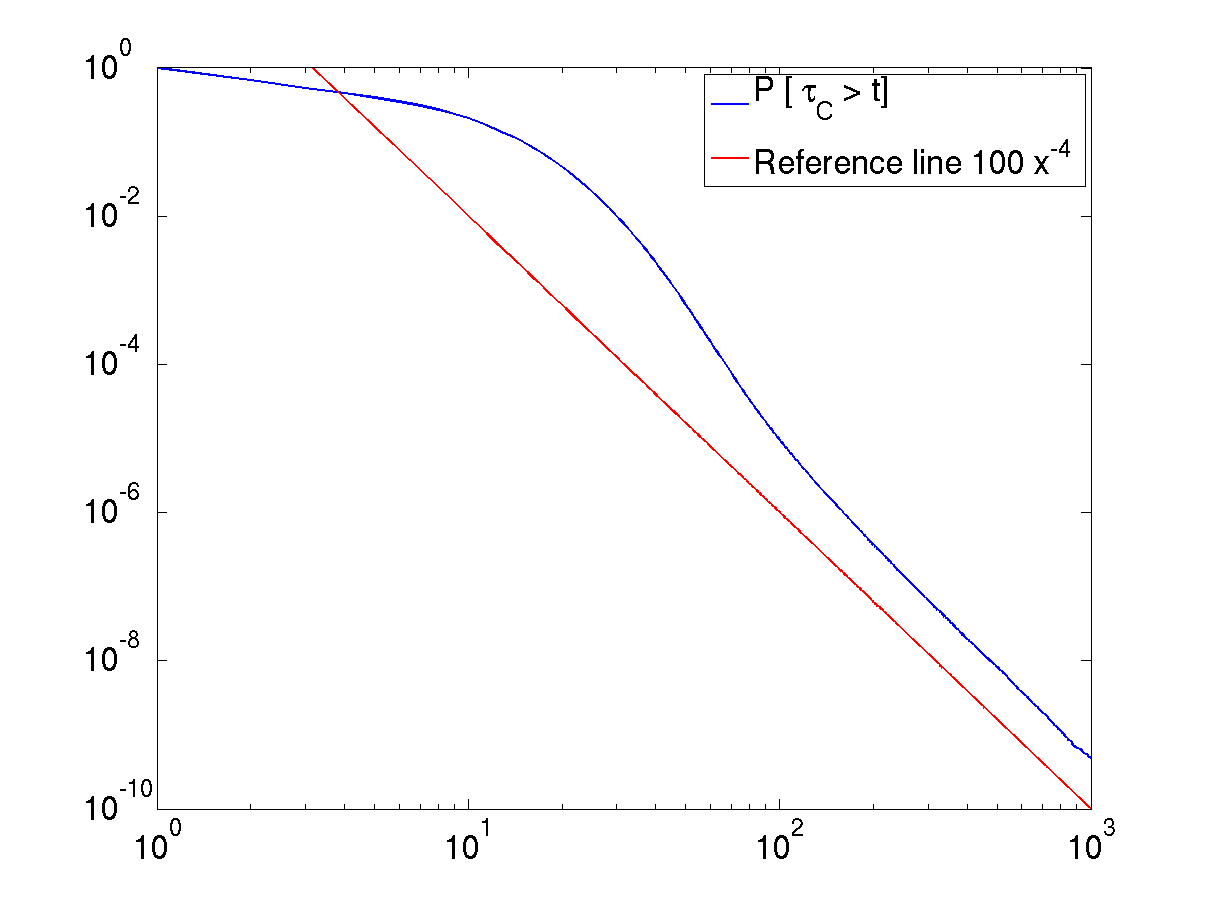}}
\caption{$\mathbb{P}_{E^{*}}[ \tau_{\mathfrak{C}} > t]$ versus $t$ when
  starting from $E^{*}$ in a log-log plot. Red line is a reference
  line with slope $-4$. Model parameters are $T_{L} = 1, T_{R} = 2$, $N = 3$, and $M =
  2$. Sample size of Monte
  Carlo simulation is $10^{10}$.} 
\label{SEE_max}
\end{figure}

\subsection{Main conclusions.} 

The previous subsection verifies two numerical conditions {\bf(N1)} and {\bf(N2)} for
$\mathbf{E}_t$ with parameter $2M$. The slopes of $\mathbb{P}_{\mathbf{E}^{*}}[
 \tau_{\mathfrak{C}} > t]$ and $\mathbb{P}_{\pi}[
 \tau_{\mathfrak{C}} > t]$ in the log-log plot are $2M$ and $2M - 1$ respectively. 

\medskip

We also need the uniqueness of $\pi$.

\begin{pro}
For any $h > 0$, $\mathbf{E}^{h}_{n}$ admits at most one invariant
probability measure.
\end{pro}
\begin{proof}
By the proof of Theorem \ref{urs1}, for any $\mathbf{E} \in K$,
$P^{h/2}( \mathbf{E}, \cdot)$ has strictly positive density on $K$. In
addition, $P^{h/2}(\mathbf{E}_{0}, K) > 0$ for any $\mathbf{E}_{0} \in
\mathbb{R}^{N}_{+}$. Hence $P^{h}(\mathbf{E}_{0}, \cdot)$ has positive
density on $K$. This implies that every $\mathbf{E}_{0} \in
\mathbb{R}^{N}_{+}$ belongs to the same ergodic component. Therefore,
$\mathbf{E}^{h}_{n} $ cannot have more than one invariant probability measure.
\end{proof}

\bigskip 

In summary, we have the following conclusions for
$\mathbf{E}_{t}$. Since now $P_{x}[ \tau_{\mathfrak{C}} > t]$ and
$P_{\pi}[ \tau_{\mathfrak{C}} > t]$ have different tails, we can apply
conclusions (a) - (e) with $\beta = 2M$ when $\pi$ is not
involved, and $\beta = 2M-1$ if the initial distribution is $\pi$.

\begin{enumerate}
  \item For any $T_{L}$, T$_{R}$, there exists a unique invariant probability measure $\pi$,
    i.e., the nonequilibrium steady-state, which is absolutely
    continuous with respect to the Lebesgue measure on
    $\mathbb{R}^{N}_{+}$. 
\item For almost every $\mathbf{E}_{0} \in \mathbb{R}^{N}_{+}$ and any
  sufficiently small $\epsilon > 0$, we have
\begin{equation}
  \label{eq4-17}
\lim_{t\rightarrow \infty}t^{2M - 1 - \epsilon} \|
  \delta_{\mathbf{E}_{0}} P^{t} - \pi \|_{TV} = 0 \,.
\end{equation}
\item For any functions $\eta$, $\xi \in L^{\infty}(
  \mathbf{R}^{N}_{+})$, we have correlation decay rate
\begin{equation}
  \label{eq4-18}
C_{\mu}^{\eta, \xi}(t) \leq O(1) \cdot t^{\epsilon - 2M}
\end{equation}
for any $\epsilon > 0$ and $\mu$ satisfies {\bf (N1)}.
\item  For almost every points $\mathbf{E}_{0} , \mathbf{E}_{1} \in \mathbb{R}^{N}_{+}$ and any
  sufficiently small $\epsilon > 0$, we have
\begin{equation}
  \label{eq4-19}
\lim_{t\rightarrow \infty}t^{2M  - \epsilon} \|
  \delta_{\mathbf{E}_{0}} P^{t} - \delta_{\mathbf{E}_{1}} P^{t} \|_{TV} = 0 \,.
\end{equation}
\end{enumerate}

\subsection{Ergodicity of the billiard model}
The ergodicity of the billiard model is extremely difficult either
to prove or to compute. Let $\Phi_{t}$ be the flow of the billiard model,
$\mu$ be the initial measure, $\eta$ and $\xi$ be two
observables. Theoretically the decay of correlation 
\begin{equation}
  \label{eq4-20}
C^{\mu}_{\xi, \eta}(t) = |\int_{X} \xi(\Phi_{t}(x)) \eta(x) \mu(\mathrm{d}x)
  - \int_{X} \xi(\Phi_{t}(x)) \mu( \mathrm{d}x) \int_{X}\eta(x)
  \mu(\mathrm{d}x) |
\end{equation}
is computable. The speed of decay of correlation gives the ergodicity
of the billiard model. However, for large $t$, $C^{\mu}_{\xi,
  \eta}(t)$ has very small expectation and $O(1)$ variance. In order
to control the relative error, the sample size of Monte Carlo
simulation needs to be very large. In particular, the polynomial tail
usually can only be captured for large $t$. Simple calculation shows
that the required sample size can easily exceed the ability
of today's computer. See our discussion in \cite{2015ChaosLi} for the detail.

Instead, we choose to present the other evidence to support the polynomial
speed of correlation decay for the billiard model. The assumption is that when the total
kinetic energy in both cells are sufficiently high, the decay of
correlation is exponentially fast. Therefore, if the first passage
time distribution
to such a high energy state has a polynomial tail $\sim t^{-\beta}$, we
expect the decay rate of correlation to be also $\sim
t^{-\beta}$. Although a rigorous justification for this assumption is
not possible, this approach can be rigorously proved for simpler
deterministic dynamical systems and Markov chains. This is called the
``induced chain method'', in which we study the induced Markov chain
generated by a set such that the induced chain has exponentially fast mixing. We refer readers to \cite{li2016polynomial} for the
induced chain method for Markov processes and \cite{young1998statistical,
young1999recurrence} for the Young
towers for deterministic dynamical systems. 

In Figure \ref{tail4}, we show the tail distribution of the first passage time to
the high energy state 
\begin{equation}
  \label{eq4-21}
A = \{ ( \mathbf{x}_{1}^{1}, \mathbf{v}^{1}_{1}, \mathbf{x}_{2}^{1},
  \mathbf{v}_{2}^{1}),  ( \mathbf{x}_{1}^{2}, \mathbf{v}^{2}_{1}, \mathbf{x}_{2}^{2},
  \mathbf{v}_{2}^{2}) \,|\, | \mathbf{v}^{i}_{1}|^{2} + |
  \mathbf{v}^{i}_{2}|^{2} \geq 0.2, i = 1, 2 \}\subset \mathbf{\Omega}
\end{equation}
for a 2-cell 4-particle model as seen in Figure \ref{fig:TwoCells}. The total kinetic energy in the system
is set to be $1$. Since the rate $\sqrt{ \min\{E_{i}, E_{i+1}\} }$
only occurs when one of the total cell energy is sufficiently small
(less than $0.01$ in our case), we need some importance sampling to
reduce the computational cost. The initial cell total cell energy is
sampled from the distribution of post-collision total cell energy,
conditioning with the event that the left cell energy is less than $0.001$. We can
see that the tail of first passage time to the high energy set is
$\sim t^{-4}$. This supports our claim that the decay rate of
correlation should be $t^{-2M}$ if the number of particles in each
cell is $M$. 

\begin{figure}[htbp]
\label{tail4}
\centerline{\includegraphics[width = \linewidth]{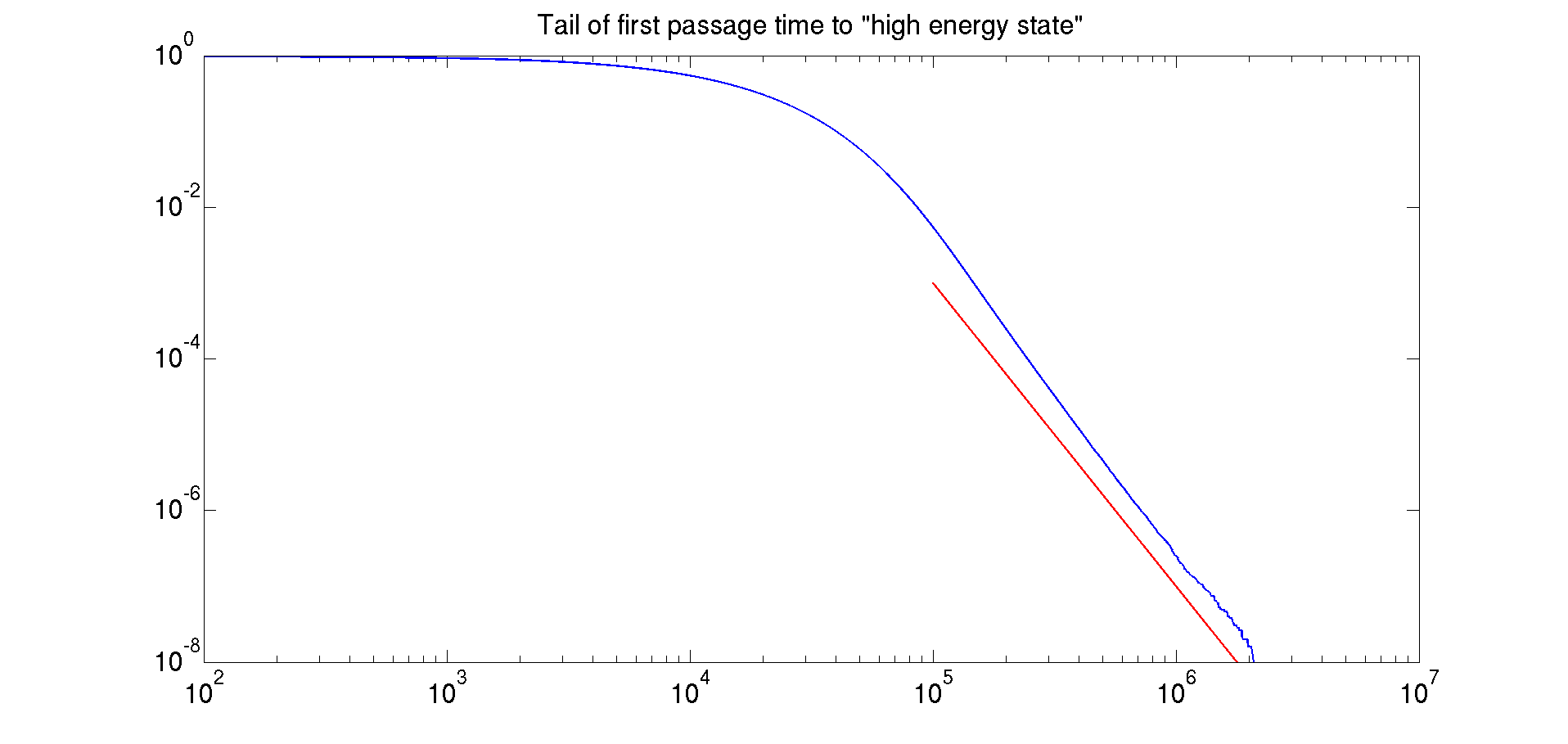}}
\caption{ Blue: Tail of first passage time $\mathbb{P}[ \tau_{A} >
  t]$ for the billiard model. Red: reference line with slope
  $-4$. Sample size of the simulation is $10^{9}$. The initial
  distribution is a conditional Liouville measure conditioning on the
  left cell energy is less than $0.001$. }
\end{figure}

\section{Comparison of thermal conductivity}
It remains to compare the thermal conductivity of the billiard model
and that of the stochastic energy exchange model. It has been reported in
\cite{gaspard2008heat2} that the billiard model has a ``normal'' thermal
conductivity, i.e., the thermal conductivity is proportional to the
reciprocal of the length of the chain. We use Monte Carlo simulations to verify that the
stochastic energy exchange model also has the ``normal'' thermal
conductivity.

We define the empirical thermal conductivity in the following
way. Consider a stochastic energy exchange model with $N$ sites and
boundary temperatures $T_{L}$ and $T_{R}$ respectively. We take the convention that
sites $0$ and $N+1$ are the left and the right heat baths
respectively. Let $J(t_{i}, k)$ be the energy flux from right to left if one
energy exchange occurs between site $k$ and site $k+1$ at time
$t_{i}$. More precisely, we have
\begin{equation}
  \label{eq5-1}
  J(t_{i}, k) = \left \{
\begin{array}{ccc}
 E_{k}' - E_{k} & \mbox{if} & k \neq 0\\
E_{k+1} - E'_{k+1} & \mbox{if} & k = 0 \,,
\end{array}
\right . 
\end{equation}
if $E_{k}$ and $E_{k+1}$ exchanges energy at $t_{i}$, where $E_{k}'$
and $E_{k+1}'$ denote the post-exchange energy. When starting from the
invariant probability measure $\pi$, the {\it thermal conductivity} is
defined as
\begin{equation}
  \label{eq5-2}
  \kappa = \lim_{T \rightarrow \infty}\frac{1}{T}
  \frac{1}{T_{R} - T_{L}} \frac{1}{N + 1} \sum_{t_{i} < T} J(t_{i}, k) \,.
\end{equation}
In other words, the thermal conductivity measures the average energy
flux between each two sites in the chain. 

In the numerical simulation, we fix boundary temperatures as $T_{L} =
1$ and $T_{R} = 2$. The thermal conductivity is then computed for
increasing $N$. Figure \ref{conductivity} shows the plot of $\kappa$ vs
$1/N$. The least square curve fitting of the plot in Figure \ref{conductivity}
gives a linear relation
$$
  \kappa(N) = \frac{0.0752}{N} + 4.02 \times 10^{-5} \,.
$$
We believe this numerical result confirms that $\kappa$ is
proportional to $1/N$. 

\begin{figure}[htbp]
\centerline{\includegraphics[width = \linewidth]{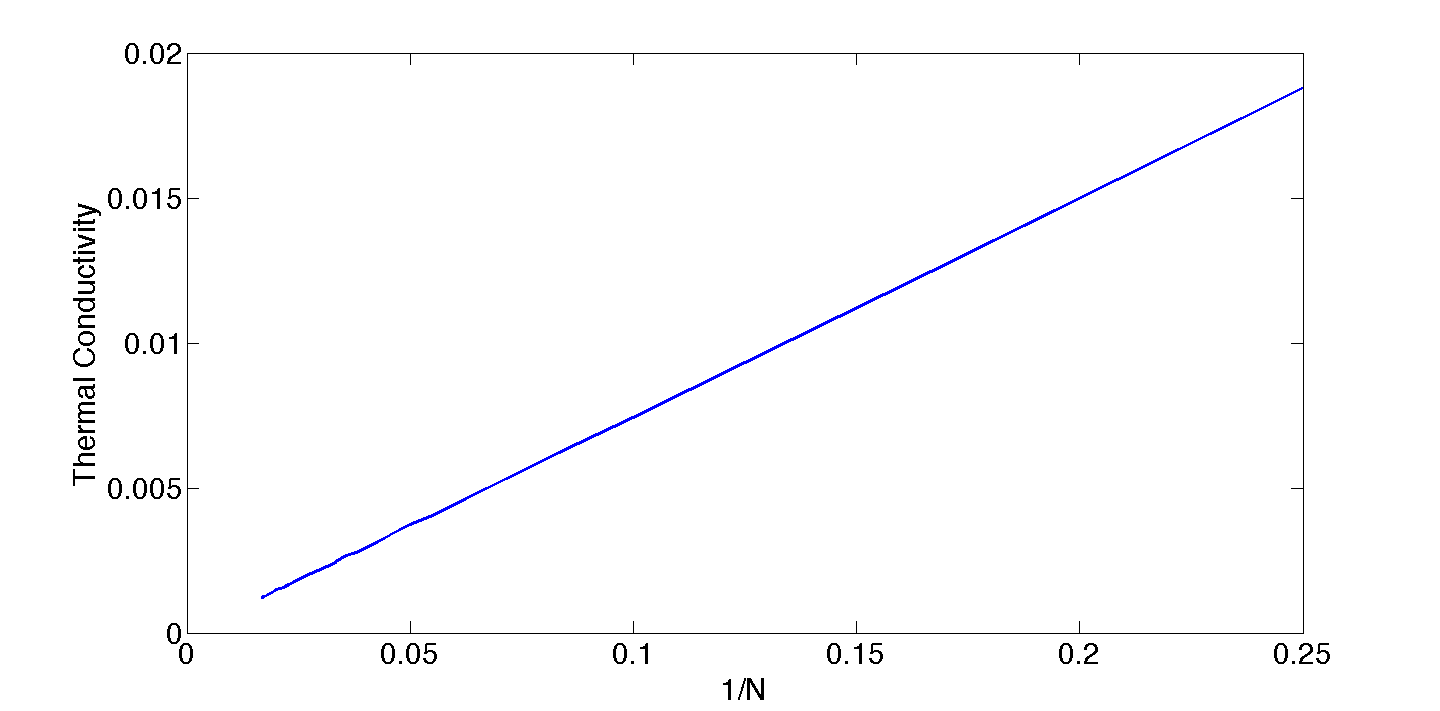}}
\caption{Thermal conductivity for the stochastic energy exchange
  model with $M = 2$. Boundary temperatures are $T_{L} = 1$ and $T_{R} = 2$. The
  length of the chain increases from $4$ to $60$. The thermal conductivity
  is computed by averaging $80$ results of \eqref{eq5-2} for $T =
60000$.}
\label{conductivity}
\end{figure}

\section{Conclusion}
In this paper we study a nonequilibrium billiard model that
mimics the dynamics of gas particles in a long and thin tube. Due to
the significant difficulty of working on the deterministic interacting
particle system
directly, we carry out a series of numerical simulations to study the
stochastic rule of energy exchanges between cells, which is
essentially given by collision events that involves particles
from neighboring cells. The time distribution of such
events and the post-collision energy distribution are studied. Numerical
results show the time evolution of the energy profile of the
nonequilibrium billiard model is approximated by a much simpler stochastic
energy exchange model. We remark that approximating a difficult
chaotic billiard system by a more mathematically tractable stochastic
process is a very generic strategy, which can be potentially
applied to other highly chaotic billiard-like systems in physics. For example, it
is known that Fermi acceleration can be found in many chaotic
billiards \cite{lenz2008tunable, lenz2010phase, leonel2010suppressing,
loskutov1999mechanism}. And the rate of energy growth is found to be
significantly larger
in many chaotic billiards or stochastic acceleration models
\cite{karlis2006hyperacceleration, karlis2007fermi}. 

We then compare the stochastic energy exchange model
and the original billiard system. A series of analytical and numerical
studies are carried out to study the ergodicity of these models. The conclusion is that the key
dynamical properties of the nonequilibrium billiard model is preserved by the stochastic
energy exchange model. Both systems have polynomial ergodicity with a
speed of correlation decay $O(t^{-2M})$, where $M$ is the
number of particles in a cell. In addition, the thermal conductivity of both models
is proportional to $1/N$. Simulation algorithms used in this paper
are the stochastic simulation algorithm (SSA)
\cite{gillespie1977exact, li2015fast} and the event-driven billiard
simulation algorithm \cite{lubachevsky1991simulate} for the stochastic model and the billiard
model, respectively. 

This result opens the door of many further investigations, as the stochastic
energy exchange model is tractable for many rigorous
studies. For example, the polynomial ergodicity can be rigorously proved by using
the same technique developed in \cite{li2016polynomial}. In addition
to the ergodicity, the mesoscopic limit problem is also worth to study. When the number of
particles in a cell is large, each collision will only
exchange a small amount of energy. Hence the stochastic energy
exchange model (after a time rescaling) has interesting slow-fast
dynamics. Such slow-fast dynamics can be approximated by a stochastic differential equation (the mesoscopic
limit equation). Many macroscopic thermodynamic properties can be
further derived from the mesoscopic limit equation. 

This paper serves as the first paper of a sequel that aims to connect
billards-like deterministic dynamics and macroscopic thermodynamic
laws. In our forthcoming papers, we will rigorously address the ergodicity, 
mesoscopic limit, and macroscopic thermodynamic properties of the stochastic energy
exchange model derived in this paper.

\end{document}